\documentclass{amsart}
\usepackage{amsmath}
  \usepackage{paralist}
  \usepackage{amssymb}
 \usepackage{amsthm}
\usepackage[colorlinks=true]{hyperref}
\hypersetup{urlcolor=blue, citecolor=red}

\newtheorem{theorem}{Theorem}[section]
\newtheorem{corollary}[theorem]{Corollary}
\newtheorem{lemma}[theorem]{Lemma}

\theoremstyle{definition}

\newtheorem{remark}[theorem]{Remark}

\numberwithin{equation}{section}

\title[]
      {Choquard equations with critical nonlinearities}

\author[Xinfu Li, Shiwang Ma]{}


\thanks{*Corresponding author\\
Email Address:  lxylxf@tjcu.edu.cn (XL); shiwangm@nankai.edu.cn
(SM)}

\begin{document}
\maketitle

\centerline{\scshape Xinfu Li$^\mathrm{a}$ and  Shiwang
Ma$^\mathrm{b*}$}
\medskip
{\footnotesize
 \centerline{$^\mathrm{a}$School of Science, Tianjin University of Commerce, Tianjin 300134,
 China} \centerline{$^\mathrm{b}$School of Mathematical
Sciences and LPMC, Nankai University, Tianjin 300071, China} }

\bigskip

\begin{abstract}
In this paper, we study the Brezis-Nirenberg type problem for
Choquard equations in $\mathbb{R}^N$
\begin{equation*}
-\Delta u+u=(I_{\alpha}\ast|u|^{p})|u|^{p-2}u+\lambda|u|^{q-2}u
\quad \mathrm{in}\  \mathbb{R}^N,
\end{equation*}
where $N\geq 3,\ \alpha\in(0,N)$, $\lambda>0$, $q\in
(2,\frac{2N}{N-2}]$, $p=\frac{N+\alpha}{N}$ or
$\frac{N+\alpha}{N-2}$ are the critical exponents in the sense of
Hardy-Littlewood-Sobolev inequality and $I_\alpha$ is the Riesz
potential. Based on the results of the subcritical problems, and by
using the subcritical approximation and the Poho\v{z}aev constraint
method, we obtain a positive and radially nonincreasing groundstate
solution in $H^1(\mathbb{R}^N)$ for the problem. To the end, the
regularity and the Poho\v{z}aev identity of solutions
to a general Choquard equation are obtained.
\end{abstract}

\section{Introduction and main results}

\setcounter{section}{1}
\setcounter{equation}{0}

In this paper, we study the following Choquard equation with
critical nonlinearities
\begin{equation}\label{e1.2222}
-\Delta u+u=(I_{\alpha}\ast|u|^{p})|u|^{p-2}u+\lambda|u|^{q-2}u
\quad \mathrm{in}\  \mathbb{R}^N,
\end{equation}
where $N\geq 3,\ \alpha\in(0,N)$, $p=\frac{N+\alpha}{N}$ or
$\frac{N+\alpha}{N-2}$, $q\in (2,\frac{2N}{N-2})$,  $\lambda>0$ and
$I_{\alpha}$ is the Riesz potential defined for every $x\in
\mathbb{R}^N \setminus \{0\}$ by
\begin{equation}\label{e1.37}
I_{\alpha}(x)=\frac{A_\alpha(N)}{|x|^{N-\alpha}}
\end{equation}
with
$A_\alpha(N)=\frac{\Gamma(\frac{N-\alpha}{2})}{\Gamma(\frac{\alpha}{2})\pi^{N/2}2^\alpha}$
and $\Gamma$ denoting the Gamma function (see \cite{Riesz1949AM},
P.19).

Problem (\ref{e1.2222}) is referred as the Brezis-Nirenberg type
problem for Choquard equations in $\mathbb{R}^N$. In the pioneering
work of Brezis and Nirenberg \cite{Brezis-Nirenberg 1983}, authors
studied the critical problem
\begin{equation}\label{e1.1111}
-\Delta u=|u|^{2^*-2}u+\lambda u \quad \mathrm{in}\ \Omega,\quad
u=0\  \mathrm{on}\ \partial\Omega,
\end{equation}
where $2^*=\frac{2N}{N-2}$ is the critical Sobolev exponent for the
embedding of $H^1_0(\Omega)$ into $L^p(\Omega)$. They proved the
existence of solutions for $\lambda>0,\ N>4$ by analyzing the local
Palais-Smale sequences below the first critical level. Since then,
there has been a considerable number of papers on problem
(\ref{e1.1111}) for the existence of positive solutions and
sign-changing solutions (see \cite{Struwe 2008} and \cite{Willem
1996}). In \cite{Brezis-Nirenberg 1983}, they also considered the
problem
\begin{equation}\label{e1.1112}
-\Delta u=|u|^{2^*-2}u+\lambda |u|^{q-2}u \quad \mathrm{in}\
\Omega,\quad u=0\  \mathrm{on}\ \partial\Omega,
\end{equation}
where $2<q<2^*$. When $N\geq 4$, they obtained that problem
(\ref{e1.1112}) has a positive solution for every $\lambda > 0$.
When $N = 3$, problem (\ref{e1.1112}) is  much more delicate: if $4
< q < 6$, problem (\ref{e1.1112}) has a positive solution for every
$\lambda > 0$; if $2 < q \leq 4$, it is only for large values of
$\lambda$ that (\ref{e1.1112}) possesses a positive solution. 

For
the Schr\"{o}dinger equation in $\mathbb{R}^N$
\begin{equation*}
-\Delta u+u=|u|^{p-2}u  \quad \mathrm{in}\ \mathbb{R}^N,\quad u\in
H^1(\mathbb{R}^N).
\end{equation*}
If $p\in (2, 2^*)$, there exists an unique positive groundstate
solution, which is radially symmetric and radially nonincreasing. If
$p\geq 2^*$, there are no nontrivial solutions. See \cite{Struwe
2008} and \cite{Willem 1996}. For the  Brezis-Nirenberg type problem
for the Schr\"{o}dinger equation in $\mathbb{R}^N$
\begin{equation}\label{e1.1114}
-\Delta u+u=|u|^{2^*-2}u+\lambda |u|^{q-2}u  \quad \mathrm{in}\
\mathbb{R}^N,\quad u\in H^1(\mathbb{R}^N),
\end{equation}
where $N\geq 3$, $2<q<2^*$ and $\lambda>0$ is a constant. The
authors in \cite{Alves-Souto-Montenegro 2012} \cite{Liu-Liao-Tang
2017} \cite{Zhang-Zou 2012} \cite{Zhang-Zou 2014}  obtained that
(\ref{e1.1114}) admits a positive ground state solution which is
radially symmetric if $N\geq4, q\in(2,2^*)$ or $N=3, q\in(2,2^*)$
and $\lambda$ is large enough.

As for the Choquard equation, the Hardy-Littlewood-Sobolev
inequality implies that
$\int_{\mathbb{R}^N}(I_\alpha\ast|u|^p)|u|^p$ is well defined for
$u\in H^1(\mathbb{R}^N)$ if $p\in
[\frac{N+\alpha}{N},\frac{N+\alpha}{N-2}]$. In 2013, Moroz and
Schaftingen \cite{Moroz-Schaftingen JFA 2013} established the
existence, qualitative properties and decay estimates of
groundstates of
\begin{equation}\label{e1.1117}
-\Delta u+u=(I_{\alpha}\ast|u|^{p})|u|^{p-2}u \quad \mathrm{in}\
\mathbb{R}^N
\end{equation}
for  $p\in \left(\frac{N+\alpha}{N},\frac{N+\alpha}{N-2}\right)$.
However, they showed that the following critical problems
\begin{equation*}
-\Delta
u+u=(I_{\alpha}\ast|u|^{\frac{N+\alpha}{N}})|u|^{\frac{N+\alpha}{N}-2}u
\quad \mathrm{in}\ \mathbb{R}^N
\end{equation*}
and
\begin{equation*}
-\Delta
u+u=(I_{\alpha}\ast|u|^{\frac{N+\alpha}{N-2}})|u|^{\frac{N+\alpha}{N-2}-2}u
\quad \mathrm{in}\ \mathbb{R}^N
\end{equation*}
have no nontrivial solutions in $H^1(\mathbb{R}^N)$. Usually,
$\frac{N+\alpha}{N}$ is called the lower critical exponent and
$\frac{N+\alpha }{N-2}$ is the upper critical exponent for the
Choquard equation. We should point out that for $N=3,\  p=2$ and
$\alpha=2$, (\ref{e1.1117}) was investigated by Pekar in \cite{Pekar
1954} to study the quantum theory of a polaron at rest. In
\cite{Lieb 1977}, it was  applied by Choquard as an approximation to
Hartree-Fock theory of one component plasma. It also arises in
multiple particles systems \cite{Gross 1996} and quantum mechanics
\cite{Penrose 1996}.

Recently, Ao \cite{Ao 2016} considered the upper critical problem
\begin{equation}\label{e1.11110}
-\Delta
u+u=(I_{\alpha}\ast|u|^{\frac{N+\alpha}{N-2}})|u|^{\frac{N+\alpha}{N-2}-2}u+|u|^{q-2}u
\quad \mathrm{in}\  \mathbb{R}^N
\end{equation}
on the space $H^1_r(\mathbb{R}^N)$. By using the Nehari manifold
method, he obtained the following result.\\
\textbf{Theorem A1.} {\it Let $\alpha\in (0,N)$, $q\in
(2,\frac{2N}{N-2})$ for $N\geq 5$ and $q\in (3,4)$ for $N=4$, then
(\ref{e1.11110})  admits a nontrivial solution in
$H_r^1(\mathbb{R}^N)$}.

We ramark that the existence of nontrivial solutions of (\ref{e1.11110}) in the cases $N=4, q\in(2,3]$ and $N=3, q\in (2,6)$ is  still an open
problem.

Van Schaftingen and  Xia \cite{Schaftingen-Xia 2017} considered the
more general lower critical problem
\begin{equation}\label{e1.56}
-\Delta
u+u=(I_{\alpha}\ast|u|^{\frac{N+\alpha}{N}})|u|^{\frac{N+\alpha}{N}-2}u+f(u)
\quad \mathrm{in}\  \mathbb{R}^N.
\end{equation}
By using the mountain-pass lemma, the Brezis-Lieb lemma and the
concentration compactness principle, they obtained the following
result.\\
\textbf{Theorem A2.} {\it  For every $N\geq 1$ and $\alpha\in (0,N)$,
there exists $\Lambda_0>0$ such that if the function $f\in
C(\mathbb{R},\mathbb{R})$ satisfies

($f_1$) $f(t)=o(t)$ as $t\to 0$,

($f_2$) $|f(t)|\leq a (|t|+|t|^{q-1})$ for some $a>0$ and $q>2$ with
$\frac{1}{q}>\frac{1}{2}-\frac{1}{N}$,

($f_3$) there exists $\mu>2$ such that $0<\mu F(t)\leq tf(t)$ for
all $t\neq0$, where $F(t)=\int_{0}^{t}f(s)ds$,

($f_4$) $\liminf_{|t|\to 0}\frac{F(t)}{t^{4/N+2}}\geq \Lambda_0$,\\
then (\ref{e1.56}) has a groundstate solution.}

Note that $f(t)=|t|^{q-2}t$  satisfies
$(f_1)$-$(f_4)$  whenever  $q\in (2,2+\frac{4}{N}]$, but does not satisfies $(f_4)$ for $q\in
(2+\frac{4}{N},\frac{2N}{N-2})$. So they left the case $q\in
(2+\frac{4}{N},\frac{2N}{N-2})$ as open problem.

\smallskip

In this paper, we will solve the above open problems and obtain
positive groundstate solutions for (\ref{e1.2222}). For
completeness, we also consider the following equations
\begin{equation}\label{e1.222}
-\Delta
u+u=\mu(I_{\alpha}\ast|u|^{p})|u|^{p-2}u+|u|^{\frac{2N}{N-2}-2}u
\quad \mathrm{in}\  \mathbb{R}^N,
\end{equation}
and
\begin{equation}\label{e1.22222}
-\Delta
u+u=\mu(I_{\alpha}\ast|u|^{\frac{N+\alpha}{N}})|u|^{\frac{N+\alpha}{N}-2}u+\lambda|u|^{\frac{2N}{N-2}-2}u
\quad \mathrm{in}\  \mathbb{R}^N,
\end{equation}
where $N\geq 3,\ \alpha\in(0,N)$, $p\in
[\frac{N+\alpha}{N},\frac{N+\alpha}{N-2}]$,  $\mu$ and $\lambda$ are
positive constants.

The main results of this paper are as follows.

\begin{theorem}\label{thm1.3}
{\it Let $N\geq 3$,\ $\alpha\in(0,N)$, $p=\frac{N+\alpha}{N-2}$ and
$\lambda>0$. Then there is a constant $\lambda_0>0$ such that (\ref{e1.2222}) admits a positive groundstate
$u\in H^1(\mathbb{R}^N)$ which is radially symmetric and radially
nonincreasing if one of the following conditions holds:

(1) $N\geq 4$ and $q\in (2,\frac{2N}{N-2})$;

(2) $N=3$ and  $q\in (4,\frac{2N}{N-2})$;

(3) $N=3$,  $q\in (2,4]$ and $\lambda>\lambda_0$.}
\end{theorem}

\smallskip

\textbf{Remark 1.} It is obvious that Theorem \ref{thm1.3} is a sharp
improvement of the results in \cite{Ao 2016}.

\begin{theorem}\label{thm1.4}
Let $N\geq 3$,\ $\alpha\in(0,N)$, $p=\frac{N+\alpha}{N}$ and
$\lambda>0$. Then there is a constant $\lambda_1>0$ such that (\ref{e1.2222}) admits a positive groundstate
$u\in H^1(\mathbb{R}^N)$ which is radially symmetric and radially
nonincreasing if one of the following conditions holds:

(1) $q\in(2,2+\frac{4}{N})$;

(2) $q\in[2+\frac{4}{N},\frac{2N}{N-2})$ and $\lambda>\lambda_1$.
\end{theorem}

\smallskip

\textbf{Remark 2.} The result in Theorem \ref{thm1.4} is new for
$q\in(2+\frac{4}{N},\frac{2N}{N-2})$. When $q\in(2,2+\frac{4}{N}]$,
it is a special case of the results in  \cite{Schaftingen-Xia
2017}.

\begin{theorem}\label{thm1.5}
Let $N\geq 3$, $\alpha\in(0,N)$ and $\mu>0$. Then there are two  constants $\mu_0, \mu_1>0$ such that(\ref{e1.222})
admits a positive groundstate  $u\in H^1(\mathbb{R}^N)$ which is
radially symmetric and radially nonincreasing if one of the
following conditions holds:

(1) $N\geq 4$ and $p\in
(1+\frac{\alpha}{N-2},\frac{N+\alpha}{N-2})$;

(2) $N\geq 4$, $p\in (\frac{N+\alpha}{N},1+\frac{\alpha}{N-2}]$ and
$\mu>\mu_0$;

(3) $N=3$ and $p\in (2+\alpha,\frac{N+\alpha}{N-2})$;

(4) $N=3$, $p\in (\frac{N+\alpha}{N},2+\alpha]$ and $\mu>\mu_1$.
\end{theorem}

\smallskip

\textbf{Remark 3.} Recently,  by using perturbation arguments, Seok
\cite{Seok AML 2017} obtained that for fixed $p\in (1,N/(N-2))$ if
$N \geq 4$ and $p\in(2, 3)$ if $N = 3$, there exists $\alpha_0>0$
depending on $p$ such that (\ref{e1.222}) admits a radially
symmetric nontrivial solution $u_\alpha\in H^1(\mathbb{R}^N)$ for
every $\alpha\in (0,\alpha_0)$. Hence, Theorem \ref{thm1.5} is an
improvement of the results in \cite{Seok AML 2017}.

\textbf{Remark 4.}  As $\alpha\to 0$, (\ref{e1.222}) formally
reduces to
\begin{equation}\label{e1.64}
-\Delta u+u=|u|^{2p-2}u+|u|^{\frac{2N}{N-2}-2}u, \quad \mathrm{in}\
\mathbb{R}^N.
\end{equation}
It is proved in \cite{Alves 1996} that (\ref{e1.64}) admits a
positive least energy solution in $H_r^1(\mathbb{R}^N)$ if $p\in
(1,N/(N-2))$ for $N \geq 4$ or $p\in(2, 3)$ for $N = 3$. Thus,
Theorem \ref{thm1.5} may be viewed as  an generalization of the results in
\cite{Alves 1996} to Choquard equations. Moreover, we obtain a
groundstate in $H^1(\mathbb{R}^N)$.

\begin{theorem}\label{thm1.6}
Let $N\geq 3$ and $\alpha\in(0,N)$. Then there exist $\lambda_2,
\mu_2>0$ such that (\ref{e1.22222}) admits a positive groundstate
$u\in H^1(\mathbb{R}^N)$ which is radially symmetric and radially
nonincreasing if $\lambda>\lambda_2$ and $\mu>\mu_2$.
\end{theorem}

\medskip

In the end of this section, we outline the methods used in this
paper. For convenience, we set $\underline{p}=\frac{N+\alpha}{N},
\bar{p}=\frac{N+\alpha}{N-2}$ and $\bar{q}=2^{*}=\frac{2N}{N-2}$ and
consider equations (\ref{e1.2222}), (\ref{e1.222}) and
(\ref{e1.22222}) in a uniform form
\begin{equation}\label{e1.22}
-\Delta u+u=\mu(I_{\alpha}\ast|u|^{p})|u|^{p-2}u+\lambda|u|^{q-2}u
\quad \mathrm{in}\  \mathbb{R}^N,
\end{equation}
with $p\in [\underline{p},\bar{p}]$, $q\in (2,\bar{q}]$, $\mu$ and
$\lambda$ being positive constants. By the Hardy-Littlewood-Sobolev
inequality( Lemma \ref{lem HLS}) and the Sobolev embedding theorem,
the functional $J_{p,q}:\ H^{1}(\mathbb{R}^N)\to \mathbb{R}$ defined
by
\begin{equation}\label{e1.24}
J_{p,q}(u)=\frac{1}{2}\int_{\mathbb{R}^N}|\nabla u|^2
+|u|^2-\frac{\mu}{2p}\int_{\mathbb{R}^N}(I_{\alpha}\ast|u|^{p})|u|^{p}-\frac{\lambda}{q}\int_{\mathbb{R}^N}|u|^{q}
\end{equation}
 is $C^1(H^{1}(\mathbb{R}^N),\mathbb{R})$ and
\begin{align}\label{e1.25}
\langle J_{p,q}'(u),w\rangle=\int_{\mathbb{R}^N}\nabla u \nabla w +
uw-\mu\int_{\mathbb{R}^N}(I_{\alpha}\ast|u|^{p})|u|^{p-2}uw-\lambda\int_{\mathbb{R}^N}|u|^{q-2}uw
\end{align}
for every $u, w\in H^{1}(\mathbb{R}^N)$. Hence, every critical
point of $J_{p,q}$ is a weak solution of (\ref{e1.22}). A nontrivial
solution $u\in H^{1}(\mathbb{R}^N)$   is called a groundstate if
$J_{p,q}(u)=c_{p,q}^{g}$, where
\begin{align}\label{e1.2}
c_{p,q}^{g}:=\inf\{J_{p,q}(v):v\in H^{1}(\mathbb{R}^N)\setminus
\{0\} \mathrm{\ and\ } J_{p,q}'(v)=0\}.
\end{align}
To prove Theorems \ref{thm1.3}-\ref{thm1.6}, we use the subcritical
approximation and the Poho\v{z}aev constraint method, which has
already been used to Schr\"{o}dinger equation \cite{Liu-Liao-Tang
2017}. More precisely, we define
\begin{equation}\label{e1.27}
c_{p,q}=\inf\{J_{p,q}(v):v\in H^{1}(\mathbb{R}^N)\setminus \{0\}\
\mathrm{and}\ P_{p,q}(v)=0 \},
\end{equation}
where $P_{p,q}(u): H^1(\mathbb{R}^N)\to \mathbb{R}$ is the
Poho\v{z}aev functional defined by
\begin{equation*}
\begin{split}
P_{p,q}(u)= &\frac{N-2}{2}\int_{\mathbb{R}^N}|\nabla u|^2
+\frac{N}{2}\int_{\mathbb{R}^N}|u|^2\\
&\quad-\frac{\mu(N+\alpha)}{2p}\int_{\mathbb{R}^N}(I_{\alpha}\ast|u|^{p})|u|^{p}-\frac{\lambda
N}{q}\int_{\mathbb{R}^N}|u|^{q}.
\end{split}
\end{equation*}
By carefully studying the properties of $c_{p,q}$ (Section 3) and by
using a sequence of groundstates of the subcritical problems, we can
show that $c_{p,q}$ is attained for various critical problems. By
further showing that every weak solution of (\ref{e1.22}) in
$H^1(\mathbb{R}^N)$ satisfies the Poho\v{z}aev identity (Section 2),
we can show that $c_{p,q}=c_{p,q}^{g}$ and  a groundstate is
obtained. In this paper, to use this method, we have to overcome two
difficulties: (a) Obtaining the Poho\v{z}aev identity of problem
(\ref{e1.22}), which is not an easy issue in our case; (b) Finer
calculations are needed for the interaction of the nonlocal
nonlinear term and the local nonlinear term.

This paper is organized as follows. In Section 2, we consider the
regularity and the Poho\v{z}aev identity of solutions for a general
Choquard equation. In Section 3 we give some preliminaries and study
the properties of $c_{p,q}$. Section 4 is devoted to the proof of
Theorems \ref{thm1.3}-\ref{thm1.6}.

\smallskip

\textbf{Basic notations}: Throughout this paper, we assume $N\geq
3$. $ C_c^{\infty}(\mathbb{R}^N)$ denotes the space of the functions
infinitely differentiable with compact support in $\mathbb{R}^N$.
$L^r(\mathbb{R}^N)$ with $1\leq r<\infty$ denotes the Lebesgue space
with the norms
$\|u\|_r=\left(\int_{\mathbb{R}^N}|u|^r\right)^{1/r}$.
 $ H^1(\mathbb{R}^N)$ is the usual Sobolev space with norm
$\|u\|_{H^1(\mathbb{R}^N)}=\left(\int_{\mathbb{R}^N}|\nabla
u|^2+|u|^2\right)^{1/2}$. $ D^{1,2}(\mathbb{R}^N)=\{u\in
L^{\frac{2N}{N-2}}(\mathbb{R}^N): |\nabla u|\in
L^2(\mathbb{R}^N)\}$. $H_r^1(\mathbb{R}^N)=\{u\in H^1(\mathbb{R}^N):
u\  \mathrm{is\ radially \ symmetric}\}$.

\section{Regularity of solutions and Poho\v{z}aev identity}

\setcounter{section}{2}
\setcounter{equation}{0}

In this section, we consider the general Choquard equation
\begin{equation}\label{e1.1}
-\Delta u+u=(I_{\alpha}\ast F(u))f(u)+g(u)\quad \mathrm{in}\
\mathbb{R}^N,
\end{equation}
where $N\geq 3,\ \alpha\in(0,N)$, $f,\ g\in
C(\mathbb{R},\mathbb{R})$, $F(s)=\int_0^sf(t)dt$,
$G(s)=\int_0^sg(t)dt$, $f$ and $g$ satisfy the following
assumptions:

(A1) There exists a positive constant $C_1$ such that for every
$s\in \mathbb{R}$,
$$|sf(s)|\leq C_1\left(|s|^{\frac{N+\alpha}{N}}+|s|^{\frac{N+\alpha}{N-2}}\right).$$

(A2) There exists a positive constant $C_2$ such that for every
$s\in \mathbb{R}$,
$$|sg(s)|\leq C_2\left(|s|^{2}+|s|^{\frac{2N}{N-2}}\right).$$

Next, we prove that any weak solution of (\ref{e1.1}) in
$H^1(\mathbb{R}^N)$ has additional regularity properties, which
allows us to establish the Poho\v{z}aev identity for all finite
energy solutions.

\begin{theorem}\label{thm1.1}
Assume that $N\geq 3$, $\alpha\in (0,N)$, (A1) and (A2) hold. If
$u\in H^1(\mathbb{R}^N)$ is a solution of (\ref{e1.1}), then $u\in
W_{\mathrm{loc}}^{2,p}(\mathbb{R}^N)$ for every $p>1$. Moreover, $u$
satisfies the Poho\v{z}aev identity
\begin{equation}\label{e1.4} \frac{N-2}{2}\int_{\mathbb{R}^N}|\nabla
u|^2+\frac{N}{2}\int_{\mathbb{R}^N}|u|^2=\frac{N+\alpha}{2}\int_{\mathbb{R}^N}(I_\alpha\ast
F(u))F(u)+N\int_{\mathbb{R}^N}G(u).
\end{equation}
\end{theorem}

To prove Theorem \ref{thm1.1}, we follow the proof in \cite{Moroz-Schaftingen 2015} for the equation
$-\Delta u+u=(I_\alpha\ast(Hu))K$
and in \cite{Brezis and Kato 1979} for the equation $-\Delta u+u=Vu$ . To the
end, we need some lemmas. The first lemma is cited from
\cite{Moroz-Schaftingen 2015}.

\begin{lemma}\label{lem1.1}
Let $N\geq 2,\ \alpha\in (0,N)$ and $\theta\in (0,2)$. If $H, K\in
L^{\frac{2N}{2+\alpha}}(\mathbb{R}^N)+L^{\frac{2N}{\alpha}}(\mathbb{R}^N)$
and $\frac{\alpha}{N}<\theta<2-\frac{\alpha}{N}$, then for every
$\epsilon>0$, there exists a positive constant $C(\epsilon,\theta)$
such that for every $u\in H^1(\mathbb{R}^N)$,
\begin{equation*}
\int_{\mathbb{R}^N}(I_\alpha\ast(H|u|^{\theta}))K|u|^{2-\theta}\leq
\epsilon^2\int_{\mathbb{R}^N}|\nabla
u|^2+C(\epsilon,\theta)\int_{\mathbb{R}^N}|u|^2.
\end{equation*}
\end{lemma}

The following lemma can be found in  \cite{Brezis and Kato 1979}.

\begin{lemma}\label{lem1.20}
Let $N\geq 3$. If $V\in
L^{\infty}(\mathbb{R}^N)+L^{\frac{N}{2}}(\mathbb{R}^N)$, then for
every $\epsilon>0$, there exists a positive constant $C(\epsilon)$
such that for every $u\in H^1(\mathbb{R}^N)$,
\begin{equation*}
\int_{\mathbb{R}^N}V|u|^2\leq \epsilon^2\int_{\mathbb{R}^N}|\nabla
u|^2+C(\epsilon)\int_{\mathbb{R}^N}|u|^2.
\end{equation*}
\end{lemma}

By using Lemmas \ref{lem1.1} and \ref{lem1.20}, we can obtain the
following result.

\begin{lemma}\label{pro1.1}
Let $N\geq 3$ and $\alpha\in (0,N)$. If $H, K\in
L^{\frac{2N}{2+\alpha}}(\mathbb{R}^N)+L^{\frac{2N}{\alpha}}(\mathbb{R}^N)$,
$V\in L^{\infty}(\mathbb{R}^N)+L^{\frac{N}{2}}(\mathbb{R}^N)$
 and $u\in H^1(\mathbb{R}^N)$ solves
\begin{equation}\label{e1.3}
-\Delta u+u=(I_\alpha\ast(Hu))K+Vu \quad \mathrm{in}\ \mathbb{R}^N,
\end{equation}
then $u\in L^r(\mathbb{R}^N)$ for
$r\in[2,\frac{N}{\alpha}\frac{2N}{N-2})$. Moreover, there exists a
positive constant $C(p)$ independent of $u$ such that
\begin{equation*}
\|u\|_p\leq C(p)\|u\|_2.
\end{equation*}
\end{lemma}

\begin{proof}
We follow the proof of \cite{Brezis and Kato 1979} and
\cite{Moroz-Schaftingen 2015}. Set $H=H_1+H_2$, $K=K_1+K_2$ and
$V=V_1+V_2$, where $H_1, K_1\in
L^{\frac{2N}{\alpha}}(\mathbb{R}^N)$, $H_2,K_2\in
L^{\frac{2N}{2+\alpha}}(\mathbb{R}^N)$, $V_1\in
L^{\infty}(\mathbb{R}^N)$ and $V_2\in
L^{\frac{N}{2}}(\mathbb{R}^N)$. By Lemma \ref{lem1.1} with
$\theta=1$ and Lemma \ref{lem1.20}, there exists a constant
$\lambda>0$ such that for every $w\in H^1(\mathbb{R}^N)$,
\begin{equation}\label{e1.5}
\int_{\mathbb{R}^N}(I_\alpha\ast((|H_1|+|H_2|)|w|))((|K_1|+|K_2|)|w|)\leq
\frac{1}{4}\int_{\mathbb{R}^N}|\nabla
w|^2+\frac{\lambda}{4}\int_{\mathbb{R}^N}|w|^2
\end{equation}
and
\begin{equation}\label{e1.6}
\int_{\mathbb{R}^N}(|V_1|+|V_2|)|w|^2\leq
\frac{1}{4}\int_{\mathbb{R}^N}|\nabla
w|^2+\frac{\lambda}{4}\int_{\mathbb{R}^N}|w|^2.
\end{equation}
For any function $M(x)$ and each $j\in \mathbb{N}$ define $M_{j}(x)$
by
\begin{equation*}
M_{j}(x)=\left\{\begin{array}{ll}
j, &\ \mathrm{if}\ M(x)> j,\\
M(x), &\ \mathrm{if}\ |M(x)|\leq j,\\
-j, &\ \mathrm{if}\ M(x)< -j.
\end{array}
\right.
\end{equation*}
Then the sequences $\{H_j:=H_1+H_{2j}\}, \{K_j:=K_1+K_{2j}\} \in
L^{\frac{2N}{\alpha}}(\mathbb{R}^N)$, $\{V_j:=V_1+V_{2j}\} \in
L^{\infty}(\mathbb{R}^N)$ satisfy $|H_j|\leq |H_1|+|H_2|$,
$|K_j|\leq |K_1|+|K_2|$, $|V_j|\leq |V_1|+|V_2|$, $H_j\to H$,
$K_j\to K$ and  $V_j\to V$ a.e. on $\mathbb{R}^N$. For each $j$,
define a bilinear form $a_j: H^1(\mathbb{R}^N)\times
H^1(\mathbb{R}^N)\to \mathbb{R}$ by
\begin{equation*}
a_j(\varphi,\psi)=\int_{\mathbb{R}^N}\nabla\varphi\nabla
\psi+\lambda\varphi\psi-\int_{\mathbb{R}^N}(I_\alpha\ast
(H_j\varphi))K_j\psi-\int_{\mathbb{R}^N}V_j\varphi\psi,\ \varphi,
\psi\in H^1(\mathbb{R}^N).
\end{equation*}
By (\ref{e1.5}) and (\ref{e1.6}), we have
\begin{equation}\label{e1.7}
a_j(\varphi,\varphi)\geq
\frac{1}{2}\int_{\mathbb{R}^N}|\nabla\varphi|^2+\frac{\lambda}{2}\int_{\mathbb{R}^N}|\varphi|^2
\end{equation}
for any $\varphi\in H^1(\mathbb{R}^N)$ and $j\in \mathbb{N}$. That
is, $a_j$ is coercive. The Lax-Milgram theorem \cite{Brezis 2012}
implies that there exists a unique solution $u_j\in
H^1(\mathbb{R}^N)$ to the problem
\begin{equation}\label{e1.8}
-\Delta u_j+\lambda
u_j=(I_\alpha\ast(H_ju_j))K_j+V_ju_j+(\lambda-1)u \quad \mathrm{in}\
\mathbb{R}^N,
\end{equation}
where $u\in H^1(\mathbb{R}^N)$ is the given solution of
(\ref{e1.3}).

We claim that the sequence $\{u_j\}$ converges weakly to $u$ in
$H^1(\mathbb{R}^N)$ as $j\to \infty$. Indeed, multiplying both sides
of (\ref{e1.8}) by $u_j$ and integrating it over $\mathbb{R}^N$, by
using (\ref{e1.7}) and the Cauchy inequality, we obtain that
\begin{equation*}
\frac{1}{2}\int_{\mathbb{R}^N}|\nabla
u_j|^2+\frac{\lambda}{2}\int_{\mathbb{R}^N}|u_j|^2\leq
(\lambda-1)\int_{\mathbb{R}^N}uu_j\leq
\frac{\lambda}{4}\int_{\mathbb{R}^N}|u_j|^2+C(\lambda)\int_{\mathbb{R}^N}|u|^2.
\end{equation*}
That is,
\begin{equation}\label{e1.21}
\frac{1}{2}\int_{\mathbb{R}^N}|\nabla
u_j|^2+\frac{\lambda}{4}\int_{\mathbb{R}^N}|u_j|^2\leq
C(\lambda)\int_{\mathbb{R}^N}|u|^2.
\end{equation}
Hence, $\{u_j\}$ is bounded in $H^1(\mathbb{R}^N)$ and then there
exists $v\in H^1(\mathbb{R}^N)$ such that $u_j\rightharpoonup v$
weakly in $H^1(\mathbb{R}^N)$ and $u_j\to v$ a.e. on $\mathbb{R}^N$.
Since $|V_j|\leq |V_1|+|V_2|$, we have $\{V_j\}$ is bounded in
$L^{\infty}(\mathbb{R}^N)+L^{\frac{N}{2}}(\mathbb{R}^N)$ and then
$\{V_ju_j\}$ is bounded in
$L^{2}(\mathbb{R}^N)+L^{\frac{2N}{N+2}}(\mathbb{R}^N)$. Thus,
\begin{equation}\label{e1.9}
\int_{\mathbb{R}^N}V_ju_j\varphi\to\int_{\mathbb{R}^N}Vv\varphi
\end{equation}
as $j\to \infty$ for any $\varphi\in C_c^{\infty}(\mathbb{R}^N)$.
Since $|K_j|\leq |K_1|+|K_2|$ and $|H_j|\leq |H_1|+|H_2|$, we obtain
that $\{K_j\}$ and $\{H_j\}$ are bounded in
$L^{\frac{2N}{2+\alpha}}(\mathbb{R}^N)+L^{\frac{2N}{\alpha}}(\mathbb{R}^N)$
and then $\{H_ju_j\}$ and $\{K_j\varphi\}$ are bounded in
$L^{\frac{2N}{N+\alpha}}(\mathbb{R}^N)$. It follows from Lemma
\ref{lem weak} that $H_ju_j\rightharpoonup Hv$ weakly in
$L^{\frac{2N}{N+\alpha}}(\mathbb{R}^N)$ and the Lebesgue dominated
convergence theorem implies that $K_j\varphi\to K\varphi$ strongly
in $L^{\frac{2N}{N+\alpha}}(\mathbb{R}^N)$ for any $\varphi\in
C_c^{\infty}(\mathbb{R}^N)$ and then
\begin{equation}\label{e1.10}
\int_{\mathbb{R}^N}(I_\alpha\ast(H_ju_j))K_j\varphi\to
\int_{\mathbb{R}^N}(I_\alpha\ast(Hv))K\varphi
\end{equation}
as $j\to \infty$. In view of (\ref{e1.8}),(\ref{e1.9}) and
(\ref{e1.10}), $v\in H^1(\mathbb{R}^N)$ is a weak solution of
\begin{equation}\label{e1.11}
-\Delta v+\lambda v=(I_\alpha\ast(Hv))K+Vv+(\lambda-1)u \quad
\mathrm{in}\ \mathbb{R}^N.
\end{equation}
Since (\ref{e1.11}) has a unique solution $u$, we obtain that $v=u$
and the claim holds.

For $\mu>0$, we define the truncation $u_{j,\mu}: \mathbb{R}^N\to
\mathbb{R}$  by
\begin{equation*}
u_{j,\mu}(x)=\left\{\begin{array}{ll}
\mu, &\ \mathrm{if}\ u_{j}(x)\geq \mu,\\
u_{j}(x), &\ \mathrm{if}\ -\mu<u_{j}(x)<\mu,\\
-\mu, &\ \mathrm{if}\ u_{j}(x)\leq -\mu.
\end{array}
\right.
\end{equation*}
For any $p\geq 2$, $|u_{j,\mu}|^{p-2}u_{j,\mu}\in
H^1(\mathbb{R}^N)$. Taking it as a test function in (\ref{e1.8}),
one has
\begin{equation}\label{e1.12}
\begin{split}
\frac{4(p-1)}{p^2}&\int_{\mathbb{R}^N}|\nabla
(u_{j,\mu}^{p/2})|^2+\lambda \int_{\mathbb{R}^N}|
u_{j,\mu}^{p/2}|^2\\
&\leq
\int_{\mathbb{R}^N}(I_\alpha\ast(H_ju_j))K_j|u_{j,\mu}|^{p-2}u_{j,\mu}+\int_{\mathbb{R}^N}V_ju_j|u_{j,\mu}|^{p-2}u_{j,\mu}\\
&\quad \quad \quad
+(\lambda-1)\int_{\mathbb{R}^N}u|u_{j,\mu}|^{p-2}u_{j,\mu}.
\end{split}
\end{equation}
For any $p\in[2,\frac{2N}{\alpha})$, if $u_j\in L^p(\mathbb{R}^N)$,
then
\begin{equation}\label{e1.17}
\begin{split}
\int_{\mathbb{R}^N}(I_\alpha\ast(H_ju_j))&K_j|u_{j,\mu}|^{p-2}u_{j,\mu}\\
&\leq\frac{(p-1)}{p^2}\int_{\mathbb{R}^N}|\nabla
(u_{j,\mu}^{p/2})|^2+C(p)\int_{\mathbb{R}^N}|
u_{j,\mu}^{p/2}|^2+o_\mu(1),
\end{split}
\end{equation}
where $\lim_{\mu\to\infty}o_\mu(1)=0$. See \cite{Moroz-Schaftingen
2015}. By Lemma \ref{lem1.20} and the Lebesgue dominated convergence
theorem, we have
\begin{equation}\label{e1.16}
\begin{split}
\int_{\mathbb{R}^N}V_ju_j|u_{j,\mu}|^{p-2}u_{j,\mu}&\leq \int_{\{x:
|u_j(x)|\leq \mu\}}|V_j||u_{j,\mu}|^{p}+\int_{\{x: |u_j(x)|>
\mu\}}|V_j||u_{j}|^{p}\\
&\leq \int_{\mathbb{R}^N}(|V_1|+|V_2|)|u_{j,\mu}|^{p}+ j\int_{\{x:
|u_j(x)|> \mu\}}|u_j|^p\\
&\leq \frac{(p-1)}{p^2}\int_{\mathbb{R}^N}|\nabla
(u_{j,\mu}^{p/2})|^2+C(p)\int_{\mathbb{R}^N}|
u_{j,\mu}^{p/2}|^2+o_\mu(1).
\end{split}
\end{equation}
By the H\"{o}lder inequality and the Young inequality, we have
\begin{equation}\label{e1.18}
\begin{split}
\int_{\mathbb{R}^N}u|u_{j,\mu}|^{p-2}u_{j,\mu}\leq
\|u\|_p\|u_{j,\mu}\|_p^{p-1}\leq
C(p)\left(\int_{\mathbb{R}^N}|u|^p+\int_{\mathbb{R}^N}|u_{j,\mu}|^p\right).
\end{split}
\end{equation}
Inserting (\ref{e1.17})-(\ref{e1.18}) into (\ref{e1.12}), if $u_k\in
L^p(\mathbb{R}^N)$ with $p\in[2,\frac{2N}{\alpha})$,  we have
\begin{equation}\label{e1.19}
\begin{split}
\int_{\mathbb{R}^N}|\nabla (u_{j,\mu}^{p/2})|^2&\leq
C(p)\left(\int_{\mathbb{R}^N}|u|^p+\int_{\mathbb{R}^N}|u_{j,\mu}|^p\right)+o_\mu(1)\\
&\leq
C(p)\left(\int_{\mathbb{R}^N}|u|^p+\int_{\mathbb{R}^N}|u_{j}|^p\right)+o_\mu(1),
\end{split}
\end{equation}
where $C(p)$ is a positive constant independent of $\mu$ and $j$.
Letting $\mu\to \infty$ in (\ref{e1.19}) and by using the Sobolev
imbedding theorem, we have
\begin{equation}\label{e1.20}
\begin{split}
\left(\int_{\mathbb{R}^N}|u_j|^{\frac{p}{2}\frac{2N}{N-2}}\right)^{\frac{N-2}{N}}\leq
C(p)\left(\int_{\mathbb{R}^N}|u|^p+\int_{\mathbb{R}^N}|u_{j}|^p\right).
\end{split}
\end{equation}
In view of (\ref{e1.21}), we have $\|u_j\|_2\leq C(\lambda)\|u\|_2$,
where $C(\lambda)$ independent of $j$. Iterating this process from
$p=2$ a finite time of steps, we obtain finally for every
$p\in[2,\frac{2N}{\alpha})$,
\begin{equation*}
\begin{split}
\|u_j\|_{\frac{p}{2}\frac{2N}{N-2}}\leq C(p)\|u\|_2.
\end{split}
\end{equation*}
By Fatou's lemma, $\|u\|_{\frac{p}{2}\frac{2N}{N-2}}\leq
C(p)\|u\|_2$. That is, $u\in L^r(\mathbb{R}^N)$ for
$r\in[2,\frac{N}{\alpha}\frac{2N}{N-2})$ and $\|u\|_r\leq
C(r)\|u\|_2$. The proof is complete.
\end{proof}

\textbf{Proof of Theorem \ref{thm1.1}.}  Set $H(s)=F(s)/s$,
$K(s)=f(s)$ and $V(s)=g(s)/s$, then (\ref{e1.1}) can be written in
the form
\begin{equation*}
-\Delta u+u=(I_{\alpha}\ast (H(u)u))K(u)+V(u)u\quad \mathrm{in}\
\mathbb{R}^N.
\end{equation*}
By (A1) and (A2), we have
\begin{equation*}
\begin{split}
|K(u)|, |H(u)|\leq
C\left(|u|^{\frac{\alpha}{N}}+|u|^{\frac{2+\alpha}{N-2}}\right),\quad
|V(u)|\leq C\left(1+|u|^{\frac{4}{N-2}}\right).
\end{split}
\end{equation*}
Thus, $H(u), K(u)\in
L^{\frac{2N}{2+\alpha}}(\mathbb{R}^N)+L^{\frac{2N}{\alpha}}(\mathbb{R}^N)$,
$V(u)\in L^{\infty}(\mathbb{R}^N)+L^{\frac{N}{2}}(\mathbb{R}^N)$. By
Lemma \ref{pro1.1}, $u\in L^r(\mathbb{R}^N)$ for
$r\in[2,\frac{N}{\alpha}\frac{2N}{N-2})$. In view of (A1), $F(u)\in
L^{q}(\mathbb{R}^N)$ for  $q\in
[\frac{2N}{N+\alpha},\frac{N}{\alpha}\frac{2N}{N+\alpha})$. Since
$\frac{2N}{N+\alpha}<\frac{N}{\alpha}<\frac{N}{\alpha}\frac{2N}{N+\alpha}$,
we have $I_\alpha\ast F(u)\in L^{\infty}(\mathbb{R}^N)$, and thus
\begin{equation*}
|-\Delta u+u|\leq
C\left(|u|^{\frac{\alpha}{N}}+|u|^{\frac{2+\alpha}{N-2}}+|u|+|u|^{\frac{N+2}{N-2}}\right).
\end{equation*}
By the regularity theory for local problems in bounded domains, we
deduce that $u\in W_{\mathrm{loc}}^{2,p}(\mathbb{R}^N)$ for every
$p>1$. See Appendix B in \cite{Struwe 2008}.

The identity (\ref{e1.4}) can be  proved by using the truncation
argument, see (\cite{Moroz-Schaftingen 2015}, Theorem 3) for the
equation $-\Delta u+u=(I_{\alpha}\ast F(u))f(u)$ and (\cite{Willem
1996}, Appendix B) for $-\Delta u=g(u)$. The details will be
omitted.

\medskip

Applying Theorem \ref{thm1.1} to equation (\ref{e1.22}), we obtain
the following result.

\begin{corollary}\label{cor1.1}
Let $N\geq 3$, $\alpha\in (0,N)$, $p\in
[\frac{N+\alpha}{N},\frac{N+\alpha}{N-2}]$ and $q\in
[2,\frac{2N}{N-2}]$. If $u\in H^1(\mathbb{R}^N)$ is a solution of
(\ref{e1.22}), then $u$ satisfies the Poho\v{z}aev identity
\begin{align}\label{e1.23}
\frac{N-2}{2}\int_{\mathbb{R}^N}|\nabla u|^2
+\frac{N}{2}\int_{\mathbb{R}^N}|u|^2=\frac{\mu(N+\alpha)}{2p}\int_{\mathbb{R}^N}(I_{\alpha}\ast|u|^{p})|u|^{p}+\frac{\lambda
N}{q}\int_{\mathbb{R}^N}|u|^{q}.
\end{align}
\end{corollary}

\textbf{Remark.} The regularity and the Poho\v{z}aev identity of
solutions to (\ref{e1.22}) have been studied in \cite{Li-Ma-Zhang}
by using a direct bootstrap argument under some restrictions on $p$
and $q$. Our result here is a complement of \cite{Li-Ma-Zhang}.

\section{Properties of $c_{p,q}$}

\setcounter{section}{3}
\setcounter{equation}{0}

In this section, we first give some preliminaries and then study the
properties of $c_{p,q}$ defined in (\ref{e1.27}). The following well
known Hardy-Littlewood-Sobolev inequality can be found in
\cite{Lieb-Loss 2001}.

\begin{lemma}\label{lem HLS}
Let $p, r>1$ and $0<\alpha<N$ with $1/p+(N-\alpha)/N+1/r=2$. Let
$u\in L^p(\mathbb{R}^N)$ and $v\in L^r(\mathbb{R}^N)$. Then there
exists a sharp constant $C(N,\alpha,p)$, independent of $u$ and $v$,
such that
\begin{equation*}
\left|\int_{\mathbb{R}^N}\int_{\mathbb{R}^N}\frac{u(x)v(y)}{|x-y|^{N-\alpha}}\right|\leq
C(N,\alpha,p)\|u\|_p\|v\|_r.
\end{equation*}
If $p=r=\frac{2N}{N+\alpha}$, then
\begin{equation*}
C(N,\alpha,p)=C_\alpha(N)=\pi^{\frac{N-\alpha}{2}}\frac{\Gamma(\frac{\alpha}{2})}{\Gamma(\frac{N+\alpha}{2})}\left\{\frac{\Gamma(\frac{N}{2})}{\Gamma(N)}\right\}^{-\frac{\alpha}{N}}.
\end{equation*}
\end{lemma}

\begin{remark}\label{rek1.3}
By the Hardy-Littlewood-Sobolev inequality above, for any $v\in
L^s(\mathbb{R}^N)$ with $s\in(1,\frac{N}{\alpha})$, $I_\alpha\ast
v\in L^{\frac{Ns}{N-\alpha s}}(\mathbb{R}^N)$ and
\begin{equation*}
\|I_\alpha\ast v\|_{\frac{Ns}{N-\alpha s}}\leq
A_\alpha(N)C(N,\alpha,s)\|v\|_s.
\end{equation*}
\end{remark}

The following lemma is useful in concerning the uniform bound of
radial nonincreasing functions, see \cite{Berestycki-Lions 1983} for
its proof.

\begin{lemma}\label{lem jx}
If $u\in L^t(\mathbb{R}^N),\ 1\leq t<+\infty$, is a radial
nonincreasing function (i.e. $0\leq u(x)\leq u(y)$ if $|x|\geq
|y|$), then one has
\begin{equation*}
|u(x)|\leq |x|^{-N/t}\left(\frac{N}{|S^{N-1}|}\right)^{1/t}\|u\|_t,
\ x\neq 0.
\end{equation*}
\end{lemma}

The following lemma can be found in  \cite{Bogachev 2007} and
\cite{Willem 2013}.

\begin{lemma}\label{lem weak}
Let $\Omega\subset \mathbb{R}^N$ be a domain, $q\in (1,\infty)$ and
$\{u_n\}$ be a bounded sequence in $L^{q}(\Omega)$. If $u_n\to u$
a.e. on $\Omega$, then $u_n\rightharpoonup u$ weakly in
$L^q(\Omega)$.
\end{lemma}

The following fact will be frequently used in this paper.

\begin{lemma}\label{lem1.7}
Let $N\geq 3$,  $q\in [2,2N/(N-2)]$ and $u\in H^1(\mathbb{R}^N)$.
Then there exists a positive constant $C$ independent of $q$ and $u$
such that
\begin{equation*}
\|u\|_q\leq C\|u\|_{H^1(\mathbb{R}^N)}.
\end{equation*}
\end{lemma}

Next, we will show that there exists $u\in
H^1(\mathbb{R}^N)\setminus \{0\}$ such that $P_{p,q}(u)=0$. Thus,
$c_{p,q}$ is well defined. To this end, we first give an elementary
lemma, see \cite{Li-Ma-Zhang} for its proof.

\begin{lemma}\label{lem1.3}
Let $a>0,\ c>0,\ b\in \mathbb{R}$,\ $n\geq 3$ and $\alpha>0$ be
constants. Define $f:[0,\infty)\to \mathbb{R}$ as
\begin{align*}
f(t)=a t^{n-2}+b t^{n}-c t^{n+\alpha}.
\end{align*}
Then $f$ has a unique critical point which corresponds to its
maximum.
\end{lemma}

For any function $u(x)$ and $\tau\geq 0$, define $u_\tau:
\mathbb{R}^{N}\to \mathbb{R}$ by
\begin{align}\label{e1.29}
u_{\tau}(x)=\left\{\begin{array}{cc}
u(x/\tau),&\tau>0,\\
0,& \tau=0.
\end{array}
 \right.
\end{align}
Then we have the following result.

\begin{lemma}\label{lem1.2}
Let $N\geq 3$, $\alpha\in (0,N)$, $p\in[\underline{p},\bar{p}]$ and
$q\in (2,\bar{q}]$. Then for every $u\in
H^{1}(\mathbb{R}^N)\setminus \{0\}$, there exists a unique
$\tau_0>0$ such that $P_{p,q}(u_{\tau_0})=0$. Moreover,
$J_{p,q}(u_{\tau_0})=\max_{\tau\geq 0}J_{p,q}(u_{\tau})$.
\end{lemma}

\begin{proof}
By direct calculation, we have
\begin{equation}\label{e1.30}
\begin{split}
\varphi(\tau):=J_{p,q}(u_\tau)=&\frac{1}{2}\tau^{N-2}\int_{\mathbb{R}^N}|\nabla u|^2+\frac{1}{2}\tau^{N}\int_{\mathbb{R}^N}|u|^2  \\
                       &\quad \quad -\frac{\mu}{2p}\tau^{N+\alpha}\int_{\mathbb{R}^N}(I_{\alpha}\ast|u|^{p})|u|^{p}-\frac{\lambda}{q}\tau^{N}\int_{\mathbb{R}^N}|u|^{q}.
\end{split}
\end{equation}
By Lemma \ref{lem1.3}, $\varphi(\tau)$ has a unique critical point
$\tau_0$ which corresponding to its maximum. Hence,
$J_{p,q}(u_{\tau_0})=\max_{\tau\geq 0}J_{p,q}(u_{\tau})$ and
\begin{equation*}
\begin{split}
0=\varphi'(\tau_0)=&\frac{N-2}{2}\tau_0^{N-3}\int_{\mathbb{R}^N}|\nabla
u|^2+\frac{N}{2}\tau_0^{N-1}\int_{\mathbb{R}^N}|u|^2\\
&\quad
-\frac{\mu(N+\alpha)}{2p}\tau_0^{N+\alpha-1}\int_{\mathbb{R}^N}(I_{\alpha}\ast|u|^{p})|u|^{p}-\frac{\lambda
N}{q}\tau_0^{N-1}\int_{\mathbb{R}^N}|u|^{q}.
\end{split}
\end{equation*}
That is, $P_{p,q}(u_{\tau_0})=0$. The proof is complete.
\end{proof}

The following result gives a minimax description of the least energy level in the subcritical case, which is a direct consequence  of  the main results in \cite{Li-Ma-Zhang} and Corollary
\ref{cor1.1}.

\begin{corollary}\label{cor1.2}
Let $N\geq 3, \alpha\in (0,N)$, $p\in
(\frac{N+\alpha}{N},\frac{N+\alpha}{N-2})$ and $q\in
(2,\frac{2N}{N-2})$. Then for every $\mu, \lambda>0$, problem
(\ref{e1.22}) admits a positive groundstate  $u\in
H^{1}(\mathbb{R}^N)$ which is radially symmetric and radially
nonincreasing. Moreover, $c_{p,q}^g=c_{p,q}^{mp}$, where
$c_{p,q}^{g}$ is defined in (\ref{e1.2}) and 
\begin{align}\label{e1.31}
c_{p,q}^{mp}:=\inf_{\gamma\in \Gamma}\sup_{t\in
[0,1]}J_{p,q}(\gamma(t)),
\end{align}
\begin{align}\label{e1.32}
\Gamma=\{\gamma\in C([0,1],H^{1}(\mathbb{R}^N)): \gamma(0)=0\
\textrm{and} \  J_{p,q}(\gamma(1))<0\}.
\end{align}
\end{corollary}

Now we are ready to give the relationship of $c_{p,q}$ and
$c_{p,q}^g$.

\begin{lemma}\label{rek1.1}
Let $N\geq 3$ and $\alpha\in (0,N)$. Then $c_{p,q}\leq c_{p,q}^g$ for
$p\in[\underline{p},\bar{p}]$, $q\in (2,\bar{q}]$ and $c_{p,q}=c_{p,q}^g$ for
$p\in(\underline{p},\bar{p})$, $q\in (2,\bar{q})$.
\end{lemma}

\begin{proof}
The first assertion follows from Corollary \ref{cor1.1}. To prove the second assertion,
for any $u\in H^1(\mathbb{R}^N)\setminus \{0\}$ with $P_{p,q}(u)=0$, let
$u_\tau$ be defined in (\ref{e1.29}). By (\ref{e1.30}), there exists
$\tau_0>0$ large enough such that $J_{p,q}(u_{\tau_0})<0$. Corollary \ref{cor1.2} and Lemma \ref{lem1.2} imply that
\begin{equation*}
c_{p,q}^{mp}\leq \max_{\tau\geq 0}J_{p,q}(u_\tau)=J_{p,q}(u).
\end{equation*}
Since $u$ is arbitrary, we obtain that $c_{p,q}^g=c_{p,q}^{mp}\leq
c_{p,q}$ and hence $c_{p,q}=c_{p,q}^g$ for
$p\in(\underline{p},\bar{p})$ and $q\in (2,\bar{q})$.
\end{proof}

The following several  lemmas are concerned with  the properties of $c_{p,q}$.

\begin{lemma}\label{lem1.4}
Assume that $N\geq 3$, $\alpha\in(0,N)$,
$p\in[\underline{p},\bar{p}]$ and $q\in (2,\bar{q}]$. Then for every
$\mu, \lambda>0$, $c_{p,q}\geq 0$.
\end{lemma}

\begin{proof}
Let $\{u_n\}\subset H^1(\mathbb{R}^N)\setminus \{0\}$ be a sequence
satisfying $\lim_{n\to \infty}J_{p,q}(u_n)=c_{p,q}$ and
$P_{p,q}(u_n)=0$. Then we have
\begin{equation*}
\begin{split}
 J_{p,q}(u_n)&=J_{p,q}(u_n)-\frac{1}{N}P_{p,q}(u_n)\\
                     &=\frac{1}{N}\int_{\mathbb{R}^N}|\nabla
                     u_n|^2+\frac{\mu\alpha}{2Np}\int_{\mathbb{R}^N}(I_{\alpha}\ast|u_n|^{p})|u_n|^{p}\geq 0,
\end{split}
\end{equation*}
which implies that $c_{p,q}\geq 0$.
\end{proof}

\begin{lemma}\label{lem1.5}
Assume that $N\geq 3$, $\alpha\in(0,N)$,
$p\in(\underline{p},\bar{p})$ and $q\in (2,\bar{q})$. Then for every
$\mu, \lambda>0$, $\limsup_{p\to \bar{p}}c_{p,q}\leq c_{\bar{p},q}$,
$\limsup_{p\to \underline{p}}c_{p,q}\leq c_{\underline{p},q}$,
$\limsup_{q\to \bar{q}}c_{p,q}\leq c_{p,\bar{q}}$  and
$\limsup_{p\to \underline{p}, q\to\bar{q}}c_{p,q}\leq
c_{\underline{p}, \bar{q}}$.
\end{lemma}

\begin{proof}
We only prove that $\limsup_{p\to \underline{p},
q\to\bar{q}}c_{p,q}\leq c_{\underline{p}, \bar{q}}$ here. The others
can be done similarly.  For any fixed $\epsilon\in(0,1)$, there
exists $u\in H^1(\mathbb{R}^N)\setminus \{0\}$ with
$P_{\underline{p},\bar{q}}(u)=0$ such that
$J_{\underline{p},\bar{q}}(u)<c_{\underline{p},\bar{q}}+\epsilon$.
It follows from (\ref{e1.30}) that there exists $\tau_0>0$ large
enough such that $J_{\underline{p},\bar{q}}(u_{\tau_0})\leq -2$. It
follows from the Young inequality that
\begin{equation}\label{e1.33}
|u|^{p}\leq
\frac{\bar{p}-p}{\bar{p}-\underline{p}}|u|^{\underline{p}}+
\frac{p-\underline{p}}{\bar{p}-\underline{p}}|u|^{\bar{p}}\
\mathrm{and}\ |u|^q\leq \frac{\bar{q}-q}{\bar{q}-2}|u|^{2}+
\frac{q-2}{\bar{q}-2}|u|^{\bar{q}}.
\end{equation}
By the Hardy-Littlewood-Sobolev inequality and the Sobolev embedding
theorem, there exist $C_1, C_2>0$ independent of $u$, such that
\begin{equation}\label{e1.61}
\begin{split}
&\int_{\mathbb{R}^N}\left(I_{\alpha}\ast
|u|^{\underline{p}}\right)|u|^{\underline{p}}\leq C_1
\|u\|_2^{2\underline{p}}\leq C_2\|u\|_{H^1(\mathbb{R}^N)}^{2\underline{p}},\\
&\int_{\mathbb{R}^N}\left(I_{\alpha}\ast
|u|^{\bar{p}}\right)|u|^{\bar{p}}\leq C_1
\|u\|_{2N/(N-2)}^{2\bar{p}}\leq C_2\|u\|_{H^1(\mathbb{R}^N)}^{2\bar{p}},\\
& \int_{\mathbb{R}^N}\left(I_{\alpha}\ast
|u|^{\underline{p}}\right)|u|^{\bar{p}}\leq C_1
\|u\|_2^{\underline{p}}\|u\|_{2N/(N-2)}^{\bar{p}}\leq
C_2\|u\|_{H^1(\mathbb{R}^N)}^{\underline{p}+\bar{p}}.
\end{split}
\end{equation}
Then the Lebesgue dominated convergence theorem implies that
$$\frac{\mu\tau^{N+\alpha}}{2p}\int_{\mathbb{R}^N}\left(I_{\alpha}\ast
|u|^p\right)|u|^p+\frac{\lambda}{q}\tau^{N}\int_{\mathbb{R}^N}|u|^{q}$$
is continuous on $p\in[\underline{p},\bar{p}]$ and $q\in[2,\bar{q}]$
uniformly with $\tau\in [0,\tau_0]$. Hence, there exists $\delta>0$
such that $
|J_{p,q}(u_{\tau})-J_{\underline{p},\bar{q}}(u_{\tau})|<\epsilon $
for $\underline{p}<p<\underline{p}+\delta$,
$\bar{q}-\delta<q<\bar{q}$ and $0\leq \tau\leq \tau_0$, which
implies that $J_{p,q}(u_{\tau_0})\leq -1$ for all
$\underline{p}<p<\underline{p}+\delta$ and
$\bar{q}-\delta<q<\bar{q}$. Since $J_{p,q}(u_{\tau})>0$ for $\tau$
small enough and $J_{p,q}(u_0)=0$ for every
$p\in[\underline{p},\bar{p}]$ and $q\in (2,\bar{q}]$, there exists
$\tau_{p,q}\in (0,\tau_0)$ such that
$\frac{d}{d\tau}\left(J_{p,q}(u_{\tau})\right)\mid_{\tau=\tau_{p,q}}=0$
and then $P_{p,q}(u_{\tau_{p,q}})=0$. By Lemma \ref{lem1.2},
$J_{\underline{p},\bar{q}}(u_{\tau_{p,q}})\leq
J_{\underline{p},\bar{q}}(u)$. Hence,
\begin{equation*}
c_{p,q}\leq J_{p,q}(u_{\tau_{p,q}})\leq
J_{\underline{p},\bar{q}}(u_{\tau_{p,q}})+\epsilon\leq
 J_{\underline{p},\bar{q}}(u)+\epsilon<c_{\underline{p},\bar{q}}+2\epsilon
\end{equation*}
for any $\underline{p}<p<\underline{p}+\delta$ and
$\bar{q}-\delta<q<\bar{q}$. Thus, $\limsup_{p\to \underline{p},
q\to\bar{q}}c_{p,q}\leq c_{\underline{p},\bar{q}}$.
\end{proof}

Let $p\in (\underline{p},\bar{p})$ and $q\in (2,\bar{q})$.
Corollaries \ref{cor1.1} and \ref{cor1.2} and Lemma \ref{rek1.1}
imply that there exists a positive and radially nonincreasing
function sequence $\{u_{p,q}\}\subset H_r^1(\mathbb{R}^N)\setminus
\{0\}$ such that
\begin{equation} \label{e1.34}
J'_{p,q}(u_{p,q})=0,\  J_{p,q}(u_{p,q})=c_{p,q}\ \mathrm{and}\
P_{p,q}(u_{p,q})=0.
\end{equation}
For such a function sequence, we have the following result.

\begin{lemma}\label{lem xiajixian}
Assume that $N\geq 3$, $\alpha\in(0,N)$,
$p\in(\underline{p},\bar{p})$, $q\in (2,\bar{q})$ and
$\{u_{p,q}\}\subset H_r^1(\mathbb{R}^N)\setminus \{0\}$ satisfies
(\ref{e1.34}). Then for every $\mu, \lambda>0$, $\{u_{p,q}\}_{p\to
\underline{p}}$, $\{u_{p,q}\}_{p\to \bar{p}}$ , $\{u_{p,q}\}_{q\to
\bar{q}}$, $\{u_{p,q}\}_{p\to \underline{p},q\to \bar{q}}$ are
bounded in $H^1(\mathbb{R}^N)$ and
$$\liminf_{p\to \underline{p}}c_{p,q}>0,\ \liminf_{p\to \bar{p}}c_{p,q}>0,\ \liminf_{q\to \bar{q}}c_{p,q}>0,\  \liminf_{p\to \underline{p},q\to \bar{q}}c_{p,q}>0.$$
\end{lemma}

\begin{proof}
We only prove this lemma for $p\to \bar{p}$. The others can be done
similarly. By Lemma \ref{lem1.5}, for $p\to \bar{p}$, we get that
\begin{equation}\label{e1.36}
\begin{split}
c_{\bar{p},q}+1\geq c_{p,q}&=J_{p,q}(u_{p,q})-\frac{1}{2p}\langle J'_{p,q}(u_{p,q}),u_{p,q}\rangle\\
                     &=\left(\frac{1}{2}-\frac{1}{2p}\right)\int_{\mathbb{R}^N}|\nabla u_{p,q}|^2
+|u_{p,q}|^2+\left(\frac{1}{2p}-\frac{1}{q}\right)\lambda\int_{\mathbb{R}^N}|u_{p,q}|^{q}
\end{split}
\end{equation}
for $q\geq 2p$, and
\begin{equation}\label{e1.35}
\begin{split}
c_{\bar{p},q}+1\geq c_{p,q}&=J_{p,q}(u_{p,q})-\frac{1}{q}\langle J'_{p,q}(u_{p,q}),u_{p,q}\rangle\\
                     &=\left(\frac{1}{2}-\frac{1}{q}\right)\int_{\mathbb{R}^N}|\nabla u_{p,q}|^2
+|u_{p,q}|^2\\
&\quad \quad \quad
+\left(\frac{1}{q}-\frac{1}{2p}\right)\mu\int_{\mathbb{R}^N}(I_\alpha\ast|u_{p,q}|^{p})|u_{p,q}|^{p}
\end{split}
\end{equation}
for $q\leq 2p$. Thus, the sequence $\{u_{p,q}\}$ is bounded in
$H^{1}(\mathbb{R}^N)$ for $p\to \bar{p}$.

By (\ref{e1.33}), (\ref{e1.61}), the Cauchy inequality and the
Sobolev imbedding theorem, for $p\to \bar{p}$, there exists $C_3,
C_4>0$ such that
\begin{equation*}
\begin{split}
0=P_{p,q}(u_{p,q})&=\frac{N-2}{2}\int_{\mathbb{R}^N}|\nabla
u_{p,q}|^2
+\frac{N}{2}\int_{\mathbb{R}^N}|u_{p,q}|^2\\
&\quad\quad
-\frac{\mu(N+\alpha)}{2p}\int_{\mathbb{R}^N}\left(I_{\alpha}\ast
|u_{p,q}|^{p}\right)|u_{p,q}|^{p}-\frac{\lambda N}{q}\int_{\mathbb{R}^N}|u_{p,q}|^{q}\\
&\geq
C_3\|u_{p,q}\|_{H^1(\mathbb{R}^N)}^{2}-C_4\left(\|u_{p,q}\|_{H^1(\mathbb{R}^N)}^{2\underline{p}}+\|u_{p,q}\|_{H^1(\mathbb{R}^N)}^{2\bar{p}}+\|u_{p,q}\|_{H^1(\mathbb{R}^N)}^{q}\right),
\end{split}
\end{equation*}
which implies that there exists $C_5>0$ such that
\begin{equation}\label{e1.38}
\|u_{p,q}\|_{H^1(\mathbb{R}^N)}\geq C_5.
\end{equation}
Combining (\ref{e1.36}), (\ref{e1.35}) and (\ref{e1.38}), we obtain
that $\liminf_{p\to \bar{p}}c_{p,q}>0$.
\end{proof}

Lemmas \ref{lem1.5} and \ref{lem xiajixian} imply that
$c_{\bar{p},q}, c_{\underline{p},q}, c_{p,\bar{q}},
c_{\underline{p},\bar{q}}>0$. In the following, we will give the
upper estimates of $c_{\bar{p},q}, c_{\underline{p},q},
c_{p,\bar{q}}$ and $c_{\underline{p},\bar{q}}$ in four lemmas. To
the end, for any $\epsilon, \delta>0$, we define
\begin{equation}\label{e1.57}
u_\epsilon(x)=\varphi(x)U_\epsilon(x),\  v_\delta(x)=\delta^{\frac{N}{2}}V(\delta x),
\end{equation}
where $\varphi(x) \in C_c^{\infty}(\mathbb{R}^N)$ is a cut off
function satisfying: (a) $0\leq \varphi(x)\leq 1$ for any $x\in
\mathbb{R}^N$; (b) $\varphi(x)\equiv 1$ in $B_1$; (c)
$\varphi(x)\equiv 0$ in $\mathbb{R}^N\setminus \overline{B_2}$.
Here, $B_s$ denotes the ball in $\mathbb{R}^N$ of center at origin
and radius $s$.
\begin{equation*}
U_\epsilon(x)=\frac{\left(N(N-2)\epsilon^2\right)^{\frac{N-2}{4}}}{\left(\epsilon^2+|x|^2\right)^{\frac{N-2}{2}}},
\end{equation*}
where $U_1(x)$ is the extremal function of
\begin{equation*}
\begin{split}
S_\alpha:&=\inf_{ u\in
D^{1,2}(\mathbb{R}^N)\setminus\{0\}}\frac{\int_{\mathbb{R}^N}|\nabla
u|^2}{\left(\int_{\mathbb{R}^N}\left(I_{\alpha}\ast
|u|^{\bar{p}}\right)|u|^{\bar{p}}\right)^{\frac{1}{\bar{p}}}}.
\end{split}
\end{equation*}
In \cite{Gao-Yang-1}, they proved that
$S_\alpha=\frac{S}{(A_\alpha(N) C_\alpha(N))^{\frac{1}{\bar{p}}}}$,
where $A_\alpha(N)$ is defined in (\ref{e1.37}), $C_\alpha(N)$ is in
Lemma \ref{lem HLS} and
\begin{equation*}
S:=\inf_{ u\in
D^{1,2}(\mathbb{R}^N)\setminus\{0\}}\frac{\int_{\mathbb{R}^N}|\nabla
u|^2}{\left(\int_{\mathbb{R}^N}|u|^{\frac{2N}{N-2}}\right)^{\frac{N-2}{N}}}.
\end{equation*}
 $V(x)=\frac{A}{(1+|x|^2)^{N/2}}$ is the extremal
functions of $S_1$, where
\begin{equation}\label{e2.6}
S_1=\inf_{u\in H^1(\mathbb{R}^N)\setminus
\{0\}}\frac{\int_{\mathbb{R}^N}|
u|^2}{\left(\int_{\mathbb{R}^N}\left(I_{\alpha}\ast
|u|^{\underline{p}}\right)|u|^{\underline{p}}\right)^{\frac{1}{\underline{p}}}}.
\end{equation}
See \cite{Seok 2018}. In the following, we
choose $A$ such that $\int_{\mathbb{R}^N}\left(I_{\alpha}\ast
|V|^{\underline{p}}\right)|V|^{\underline{p}}=1$.

By \cite{Brezis-Nirenberg 1983} (see also \cite{Willem 1996}), we
have the following estimates.
\begin{equation}\label{e1.118}
\int_{\mathbb{R}^N}|\nabla
u_\epsilon|^2=S^{\frac{N}{2}}+O(\epsilon^{N-2}),\ N\geq 3,
\end{equation}
\begin{equation}\label{e1.119}
\int_{\mathbb{R}^N}|
u_\epsilon|^{2N/(N-2)}=S^{\frac{N}{2}}+O(\epsilon^N),\ N\geq 3,
\end{equation}
and
\begin{equation}\label{e1.120}
\int_{\mathbb{R}^N}| u_\epsilon|^2=\left\{\begin{array}{ll}
K_2\epsilon^2+O(\epsilon^{N-2}),& N\geq 5,\\
K_2\epsilon^2|\ln \epsilon|+O(\epsilon^2),& N=4,\\
K_2\epsilon+O(\epsilon^2),& N=3,
\end{array}\right.
\end{equation}
where $K_2>0$. By direct calculation, for $p\in
(\underline{p},\bar{p})$ and $q\in (2,\bar{q})$, there exists $K_1,
K_3>0$ such that
\begin{equation}\label{e1.39}
\begin{split}
\int_{\mathbb{R}^N}|u_\epsilon|^q &\geq
(N(N-2))^{\frac{N-2}{4}q}\epsilon^{N-\frac{N-2}{2}q}\int_{B_{\frac{1}{\epsilon}}(0)}\frac{1}{(1+|x|^2)^{\frac{N-2}{2}q}}dx\\
&\geq \left\{\begin{array}{ll}
K_1\epsilon^{N-\frac{N-2}{2}q},& (N-2)q>N,\\
K_1\epsilon^{N-\frac{N-2}{2}q}|\ln \epsilon|,& (N-2)q=N,\\
K_1\epsilon^{\frac{N-2}{2}q},& (N-2)q<N
\end{array}\right.
\end{split}
\end{equation}
and
\begin{equation}\label{e1.47}
\int_{\mathbb{R}^N}\left(I_{\alpha}\ast
|u_\epsilon|^{p}\right)|u_\epsilon|^{p}\geq K_3
\epsilon^{-(N-2)p+N+\alpha}.
\end{equation}
Moreover, similar as in \cite{Gao-Yang-1} and \cite{Gao-Yang-2}, by
direct computation,
\begin{equation}\label{e1.124}
\int_{\mathbb{R}^N}\left(I_\alpha\ast|u_\epsilon|^{\bar{p}}\right)
|u_\epsilon|^{\bar{p}}\geq (A_\alpha(N)
C_\alpha(N))^{\frac{N}2}S_\alpha^{\frac{N+\alpha}2}
+O(\epsilon^{\frac{N+\alpha}{2}}).
\end{equation}

\begin{lemma}\label{lem shangjie}
Let $N\geq 3$, $\alpha\in (0,N)$ and $q\in (2,\bar{q})$.

If $q\in (2,\bar{q})$ for $N\geq 4$ or  $q\in (4,\bar{q})$ for
$N=3$, then for every $\mu, \lambda>0$,
\begin{equation}\label{e1.42}
c_{\bar{p},q}<\frac{2+\alpha}{2(N+\alpha)}\mu^{-\frac{N-2}{2+\alpha}}S_\alpha^{\frac{N+\alpha}{2+\alpha}}.
\end{equation}

If $N=3,\ q\in (2,4]$, then for every $\mu>0$, there exists
$\lambda_0>0$ such that (\ref{e1.42}) holds for $\lambda>\lambda_0$.
\end{lemma}

\begin{proof}
We use $u_\epsilon$ to estimate $c_{\bar{p},q}$, where $u_\epsilon$ is defined in (\ref{e1.57}).
By Lemma \ref{lem1.2}, there exists a unique $\tau_\epsilon$ such
that $P_{\bar{p},q}((u_\epsilon)_{\tau_\epsilon})=0$ and
$J_{\bar{p},q}((u_\epsilon)_{\tau_\epsilon})=\sup_{\tau\geq
0}J_{\bar{p},q}((u_\epsilon)_{\tau})$. Thus, $c_{\bar{p},q}\leq
\sup_{\tau\geq 0}J_{\bar{p},q}((u_\epsilon)_{\tau})$. By direct
calculation, one has
\begin{equation}\label{e1.117}
\begin{split}
J_{\bar{p},q}((u_\epsilon)_{\tau})&=\frac{\tau^{N-2}}{2}\int_{\mathbb{R}^N}|\nabla
u_\epsilon|^2+\frac{\tau^{N}}{2}\int_{\mathbb{R}^N}|u_\epsilon|^2\\
&\quad\quad\quad\quad
-\frac{\mu\tau^{N+\alpha}}{2\bar{p}}\int_{\mathbb{R}^N}\left(I_{\alpha}\ast
|u_\epsilon|^{\bar{p}}\right)|u_\epsilon|^{\bar{p}}-\frac{\lambda\tau^N}{q}\int_{\mathbb{R}^N}|u_\epsilon|^q\\
&\leq
\frac{\tau^{N-2}}{2}(S^{\frac{N}{2}}+O(\epsilon^{N-2}))-\frac{\mu\tau^{N+\alpha}}{2\bar{p}}((A_\alpha(N)C_\alpha(N))^{\frac{N}{2}}S_\alpha^{\frac{N+\alpha}{2}}+O(\epsilon^{\frac{N+\alpha}{2}}))\\
&\quad\quad+\frac{\tau^N}{2}\left\{\begin{array}{ll}
K_2\epsilon^2+O(\epsilon^{N-2}),& N\geq 5,\\
K_2\epsilon^2|\ln \epsilon|+O(\epsilon^2),& N=4,\\
K_2\epsilon+O(\epsilon^2),& N=3,
\end{array}\right.\\
&\quad\quad-\frac{\lambda\tau^N}{q} \left\{\begin{array}{ll}
K_1\epsilon^{N-\frac{N-2}{2}q},& (N-2)q>N,\\
K_1\epsilon^{N-\frac{N-2}{2}q}|\ln \epsilon|,& (N-2)q=N,\\
K_1\epsilon^{\frac{N-2}{2}q},& (N-2)q<N.
\end{array}\right.
\end{split}
\end{equation}

We claim that for every fixed $\mu>0$ there exist $\tau_0, \tau_1>0$
independent of $\epsilon$ and $\lambda>0$ such that
$\tau_\epsilon\in [\tau_0, \tau_1]$ for $\epsilon>0$ small. Suppose
by contradiction that $\tau_\epsilon\to 0$ or $\tau_\epsilon\to
\infty$ as $\epsilon\to 0$. (\ref{e1.117}) implies that
$\sup_{\tau\geq 0}J_{\bar{p},q}((u_\epsilon)_{\tau})\leq 0$ as
$\epsilon\to 0$ and then $c_{\bar{p},q}\leq 0$, which contradicts
$c_{\bar{p},q}>0$. Thus, the claim holds.

By direct calculation, we obtain
\begin{equation}\label{e1.125}
\frac{\tau^{N-2}}{2}S^{\frac{N}{2}}-\frac{\mu\tau^{N+\alpha}}{2\bar{p}}(A_\alpha(N)C_\alpha(N))^{\frac{N}{2}}S_\alpha^{\frac{N+\alpha}{2}}\leq
\frac{2+\alpha}{2(N+\alpha)}\mu^{-\frac{N-2}{2+\alpha}}S_\alpha^{\frac{N+\alpha}{2+\alpha}}.
\end{equation}
For $N\geq 4$ and $q\in (2,\bar{q})$, we have $(N-2)q>N$ and $0<N-\frac{N-2}{2}q<2$. Thus, for every $\mu, \lambda>0$,
\begin{equation}\label{e1.41}
\sup_{\tau\geq
0}J_{\bar{p},q}((u_\epsilon)_{\tau})<\frac{2+\alpha}{2(N+\alpha)}\mu^{-\frac{N-2}{2+\alpha}}S_\alpha^{\frac{N+\alpha}{2+\alpha}}
\end{equation}
for $\epsilon>0$ small enough. Similarly, if $N=3$ and $q\in
(4,\bar{q})$, then for every $\mu, \lambda>0$, (\ref{e1.41}) holds
for $\epsilon>0$ small enough. If $N=3$ and $q\in (2,4]$, for every
$\mu>0$, there exists $\lambda_0>0$ and $\epsilon_0>0$ such that
(\ref{e1.41}) holds for $\lambda>\lambda_0$ and
$\epsilon=\epsilon_0$. The proof is complete.
\end{proof}

\begin{lemma}\label{lem1.8}
Let $N\geq 3,\ \alpha\in(0,N)$ and $q\in(2,\bar{q})$.

If $2<q<2+\frac{4}{N}$, then for every $\lambda, \mu>0$,
\begin{equation}\label{e1.46}
c_{\underline{p},q}<\frac{\alpha}{2(N+\alpha)}\mu^{-\frac{N}{\alpha}}S_1^{\frac{N+\alpha}\alpha}.
\end{equation}

If $q\in[2+\frac{4}{N},\bar{q})$, then for every $\mu>0$, there
exists $\lambda_1>0$ such that (\ref{e1.46}) holds for
$\lambda>\lambda_1$.
\end{lemma}

\begin{proof}
We use $v_\delta$ to estimate $c_{\underline{p},q}$, where $v_\delta$ is defined in (\ref{e1.57}).
Similarly to the proof of Lemma \ref{lem shangjie}, we have
\begin{equation}\label{e2.8}
\begin{split}
J_{\underline{p},q}((v_\delta)_{\tau})&=\frac{\tau^{N-2}}{2}\int_{\mathbb{R}^N}|\nabla
v_\delta|^2+\frac{\tau^{N}}{2}\int_{\mathbb{R}^N}|v_\delta|^2\\
&\quad\quad\quad\quad
-\frac{\mu\tau^{N+\alpha}}{2\underline{p}}\int_{\mathbb{R}^N}\left(I_{\alpha}\ast
|v_\delta|^{\underline{p}}\right)|v_\delta|^{\underline{p}}-\frac{\lambda\tau^N}{q}\int_{\mathbb{R}^N}|v_\delta|^q\\
&=\frac{\tau^{N-2}}{2}\delta^2\int_{\mathbb{R}^N}|\nabla
V|^2+\frac{\tau^{N}}{2}\int_{\mathbb{R}^N}|V|^2-\frac{\mu\tau^{N+\alpha}}{2\underline{p}}\int_{\mathbb{R}^N}\left(I_{\alpha}\ast
|V|^{\underline{p}}\right)|V|^{\underline{p}}\\
&\quad\quad\quad\quad
-\frac{\lambda\tau^N}{q}\delta^{\frac{Nq}{2}-N}\int_{\mathbb{R}^N}|V|^{q}
\end{split}
\end{equation}
and then for every $\mu>0$, there exist $\tau_2, \tau_3>0$
independent of $\delta$ and $\lambda>0$ such that $\tau_\delta\in
[\tau_2, \tau_3]$ for $\delta>0$ small. Direct calculation gives
that
\begin{equation}\label{e1.43}
\frac{\tau^{N}}{2}\int_{\mathbb{R}^N}|V|^2-\frac{\mu\tau^{N+\alpha}}{2\underline{p}}\int_{\mathbb{R}^N}\left(I_{\alpha}\ast
|V|^{\underline{p}}\right)|V|^{\underline{p}}\leq\frac{\alpha}{2(N+\alpha)}\mu^{-\frac{N}{\alpha}}S_1^{\frac{N+\alpha}\alpha}.
\end{equation}
If $2<q<2+\frac{4}{N}$, we have $\frac{Nq}{2}-N<2$.  Thus, for every
$\mu, \lambda>0$,
\begin{equation}\label{e1.44}
\sup_{\tau\geq
0}J_{\underline{p},q}((v_\delta)_{\tau})<\frac{\alpha}{2(N+\alpha)}\mu^{-\frac{N}{\alpha}}S_1^{\frac{N+\alpha}\alpha}
\end{equation}
for $\delta>0$ small enough. If
$q\in[2+\frac{4}{N},\frac{2N}{N-2})$, then for every $\mu>0$, there
exists $\lambda_1>0$ and $\delta_0>0$ such that (\ref{e1.44}) holds
for $\lambda>\lambda_1$ and $\delta=\delta_0$. The proof is
complete.
\end{proof}

\begin{lemma}\label{lem1.9}
Let $N\geq 3,\ \alpha\in(0,N)$ and $p\in(\underline{p},\bar{p})$.

If $N\geq 4,\ p\in (1+\frac{\alpha}{N-2},\bar{p})$ or $N=3,\ p\in
(2+\alpha,\bar{p})$, then for every $\lambda, \mu>0$,
\begin{equation}\label{e1.51}
c_{p,\bar{q}}<\frac{1}{N}\lambda^{-\frac{N-2}2}S^{\frac{N}{2}}.
\end{equation}

If $N\geq 4,\ p\in (\underline{p},1+\frac{\alpha}{N-2}]$ or $N=3,\
p\in (\underline{p},2+\alpha]$, then for every $\lambda>0$, there
exists $\mu_0>0$ such that (\ref{e1.51}) holds for $\mu>\mu_0$.
\end{lemma}

\begin{proof}
We use $u_\epsilon$ to estimate $c_{p,\bar{q}}$, where $u_\epsilon$ is defined in (\ref{e1.57}).
Similarly to the proof of Lemma \ref{lem
shangjie}, we have
\begin{equation}\label{e1.48}
\begin{split}
J_{p,\bar{q}}((u_\epsilon)_{\tau})&=\frac{\tau^{N-2}}{2}\int_{\mathbb{R}^N}|\nabla
u_\epsilon|^2+\frac{\tau^{N}}{2}\int_{\mathbb{R}^N}|u_\epsilon|^2\\
&\quad\quad\quad\quad
-\frac{\mu\tau^{N+\alpha}}{2p}\int_{\mathbb{R}^N}\left(I_{\alpha}\ast
|u_\epsilon|^{p}\right)|u_\epsilon|^{p}-\frac{\lambda\tau^N}{\bar{q}}\int_{\mathbb{R}^N}|u_\epsilon|^{\bar{q}}\\
&\leq
\frac{\tau^{N-2}}{2}(S^{\frac{N}{2}}+O(\epsilon^{N-2}))-\frac{\lambda\tau^{N}}{\bar{q}}(S^{\frac{N}{2}}+O(\epsilon^N))-\frac{\mu\tau^{N+\alpha}}{2p}K_3\epsilon^{-(N-2)p+N+\alpha}\\
&\quad\quad+\frac{\tau^N}{2}\left\{\begin{array}{ll}
K_2\epsilon^2+O(\epsilon^{N-2}),& N\geq 5,\\
K_2\epsilon^2|\ln \epsilon|+O(\epsilon^2),& N=4,\\
K_2\epsilon+O(\epsilon^2),& N=3,
\end{array}\right.
\end{split}
\end{equation}
and then for every fixed $\lambda>0$ there exist $\tau_4, \tau_5>0$
independent of $\epsilon$ and $\mu>0$ such that $\tau_\epsilon\in
[\tau_4, \tau_5]$ for $\epsilon>0$ small. A direct calculation shows
that
\begin{equation}\label{e1.49}
\frac{\tau^{N-2}}{2}S^{\frac{N}{2}}-\frac{\lambda}{\bar{q}}\tau^{N}S^{\frac{N}{2}}\leq
\frac{1}{N}\lambda^{-\frac{N-2}2}S^{\frac{N}{2}}.
\end{equation}

For $N\geq 4$ and $p\in (1+\frac{\alpha}{N-2},\bar{p})$, we have
$-(N-2)p+N+\alpha>0$ and $-(N-2)p+N+\alpha<2$. Thus, for every $\mu,
\lambda>0$,
\begin{equation}\label{e1.50}
\sup_{\tau\geq
0}J_{p,\bar{q}}((u_\epsilon)_{\tau})<\frac{1}{N}\lambda^{-\frac{N-2}2}S^{\frac{N}{2}}
\end{equation}
for $\epsilon>0$ small enough. For $N\geq 4$ and $p\in
(\underline{p},1+\frac{\alpha}{N-2}]$, for every $\lambda>0$, there
exists $\mu_0>0$ and $\epsilon_1>0$ such that (\ref{e1.50}) holds
for $\mu>\mu_0$ and $\epsilon=\epsilon_1$.

Similarly, for $N=3$,  if $p\in (2+\alpha,\bar{p})$, (\ref{e1.50})
holds for every $\mu, \lambda>0$. If $p\in
(\underline{p},2+\alpha]$, for every $\lambda>0$, there exists
$\mu_1>0$ and $\epsilon_2>0$ such that (\ref{e1.50}) holds for
$\mu>\mu_1$ and $\epsilon=\epsilon_2$. The proof is complete.
\end{proof}

\begin{lemma}\label{lem1.10}
Let $N\geq 3$ and $\alpha\in(0,N)$. Then for sufficiently large
$\mu$ and $\lambda>0$,
$$c_{\underline{p},\bar{q}}<\min\left\{\frac{\alpha}{2(N+\alpha)}\mu^{-\frac{N}{\alpha}}S_1^{\frac{N+\alpha}\alpha}, \frac{1}{N}\lambda^{-\frac{N-2}2}S^{\frac{N}{2}}\right\}.$$
\end{lemma}

\begin{proof}
We first use $v_\delta$ to estimate $c_{\underline{p},\bar{q}}$,
where $v_\delta$ is defined in (\ref{e1.57}). Similarly to the proof
of Lemma \ref{lem1.8},
\begin{equation*}
\begin{split}
J_{\underline{p},\bar{q}}((v_\delta)_{\tau})
&=\frac{\tau^{N-2}}{2}\delta^2\int_{\mathbb{R}^N}|\nabla
V|^2+\frac{\tau^{N}}{2}\int_{\mathbb{R}^N}|V|^2-\frac{\mu\tau^{N+\alpha}}{2\underline{p}}\int_{\mathbb{R}^N}\left(I_{\alpha}\ast
|V|^{\underline{p}}\right)|V|^{\underline{p}}\\
&\quad\quad\quad\quad
-\frac{\lambda\tau^N}{\bar{q}}\delta^{\frac{N\bar{q}}{2}-N}\int_{\mathbb{R}^N}|V|^{\bar{q}}.
\end{split}
\end{equation*}
Thus, there exist $\tau_6,\ \tau_7>0$ independent of $\mu,\lambda>1$
and $\delta>0$ such that $\tau_\delta\in [\tau_6, \tau_7]$ for
$\delta>0$ small. Hence, for sufficiently small $\delta>0$, there
exists $\lambda_0>1$ such that
$$\sup_{\tau\geq 0}J_{\underline{p},\bar{q}}((v_\delta)_{\tau})<\frac{\alpha}{2(N+\alpha)}\mu^{-\frac{N}{\alpha}}S_1^{\frac{N+\alpha}\alpha}$$
for $\lambda>\lambda_0$ and then
$c_{\underline{p},\bar{q}}<\frac{\alpha}{2(N+\alpha)}\mu^{-\frac{N}{\alpha}}S_1^{\frac{N+\alpha}\alpha}$.

 Next, we use $u_\epsilon$ to estimate
$c_{\underline{p},\bar{q}}$, where $u_\epsilon$ is defined in
(\ref{e1.57}). Similarly to the proof of Lemma \ref{lem1.9},
\begin{equation*}
\begin{split}
J_{\underline{p},\bar{q}}((u_\epsilon)_{\tau})
&\leq
\frac{\tau^{N-2}}{2}(S^{\frac{N}{2}}+O(\epsilon^{N-2}))-\frac{\lambda}{\bar{q}}\tau^{N}(S^{\frac{N}{2}}+O(\epsilon^N))-\frac{\mu\tau^{N+\alpha}}{2\underline{p}}K_3\epsilon^{-(N-2)\underline{p}+N+\alpha}\\
&\quad\quad+\frac{\tau^N}{2}\left\{\begin{array}{ll}
K_2\epsilon^2+O(\epsilon^{N-2}),& N\geq 5,\\
K_2\epsilon^2|\ln \epsilon|+O(\epsilon^2),& N=4,\\
K_2\epsilon+O(\epsilon^2),& N=3.
\end{array}\right.
\end{split}
\end{equation*}
Thus, there exist $\tau_8,\ \tau_9>0$ independent of $\mu,\lambda>1$
and $\epsilon>0$ such that $\tau_\epsilon\in [\tau_8, \tau_9]$ for
$\epsilon>0$ small. Hence, for sufficiently small $\epsilon>0$,
there exists $\mu_0>1$ such that
$$\sup_{\tau\geq 0}J_{\underline{p},\bar{q}}((u_\epsilon)_{\tau})<\frac{1}{N}\lambda^{-\frac{N-2}2}S^{\frac{N}{2}}$$
for $\mu>\mu_0$ and then
$c_{\underline{p},\bar{q}}<\frac{1}{N}\lambda^{-\frac{N-2}2}S^{\frac{N}{2}}$.
 The proof is complete.
\end{proof}

\section{Proofs of the main results}

\setcounter{section}{4}
\setcounter{equation}{0}

In this section, we give the proofs of Theorems
\ref{thm1.3}-\ref{thm1.6}.

\medskip

\textbf{Proof of Theorem \ref{thm1.3}}. Let $\mu=1$, $p_n\to
\bar{p}^{-}$ as $n\to \infty$ and $\{u_n:=u_{p_n,q}\}\subset
H_r^1(\mathbb{R}^N)$ be a positive and radially nonincreasing
sequence which satisfies (\ref{e1.34}). By Lemma \ref{lem
xiajixian}, $\{u_n\}$ is bounded in $H^1(\mathbb{R}^N)$. Thus, there
exists a nonnegative and radially nonincreasing function $u\in
H_r^1(\mathbb{R}^N)$ such that up to a subsequence,
$u_n\rightharpoonup u$ weakly in $H^1(\mathbb{R}^N)$, $u_n\to u$
strongly in $L^t(\mathbb{R}^N)$ for $t\in (2,\frac{2N}{N-2})$ and
$u_n\to u$ a.e. in $\mathbb{R}^N$.

Next, we will show that $u$ is a solution of (\ref{e1.2222}). By
Lemma \ref{lem1.7}, $p_n\to \bar{p}^{-}$ and the H\"{o}lder
inequality, we have
\begin{equation}\label{e1.63}
\{|u_n|^{p_n}\}\ \textrm{is\ bounded\ in}\
L^{\frac{2N}{N+\alpha}}(\mathbb{R}^N),\ \{|u_n|^{q-2}u_n\}\ \textrm{is\ bounded\ in}\
L^{\frac{q}{q-1}}(\mathbb{R}^N),
\end{equation}
$\{|u_n|^{p_n-2}u_n\}$ is bounded in
$L^{\frac{2N\bar{p}}{(\bar{p}-1)(N+\alpha)}}(\mathbb{R}^N)$,
$\{|u_n|^{p_n-2}u_n\psi\}$ is bounded in
$L^{\frac{2N}{N+\alpha}}(\mathbb{R}^N)$ and
$\{|u|^{\bar{p}-2}u\psi\}\in L^{\frac{2N}{N+\alpha}}(\mathbb{R}^N)$
for any $\psi\in C_c^{\infty}(\mathbb{R}^N)$. By (\ref{e1.63}) and
Lemma \ref{lem weak}, we have $|u_n|^{p_n}\rightharpoonup
|u|^{\bar{p}}$ weakly in $L^{\frac{2N}{N+\alpha}}(\mathbb{R}^N)$ and $|u_n|^{q-2}u_n\rightharpoonup |u|^{q-2}u$ weakly in $L^{\frac{q}{q-1}}(\mathbb{R}^N)$. By
Remark \ref{rek1.3}, $I_\alpha\ast(|u|^{\bar{p}-2}u\psi)\in
L^{\frac{2N}{N-\alpha}}(\mathbb{R}^N)$.  Hence,
\begin{equation*}
\int_{\mathbb{R}^N}\left(I_{\alpha}\ast|u_n|^{p_n}\right)(|u|^{\bar{p}-2}u\psi)
\to\int_{\mathbb{R}^N}\left(I_{\alpha}\ast|u|^{\bar{p}}\right)(|u|^{\bar{p}-2}u\psi)
\end{equation*}
and
\begin{equation*}
\int_{\mathbb{R}^N}|u_n|^{q-2}u_n\psi\to
\int_{\mathbb{R}^N}|u|^{q-2}u\psi
\end{equation*}
as $n\to \infty$. It follows from $N\geq 3$ that
$\frac{N}{\frac{N-2}2(\underline{p}-1)}$ and
$\frac{N}{\frac{N-2}2(\bar{p}-1)} \in (\frac{2N}{N+\alpha},\infty)$.
Since  $p_n\to \bar{p}^{-}$ and $\psi \in L^{r}(\mathbb{R}^N)$ for
$r\in (1,\infty)$, by the Young inequality, the H\"{o}lder
inequality and Lemma \ref{lem jx} with $t=2N/(N-2)$, there exists a
constant $C>0$ independent of $n$ such that
\begin{equation}\label{e1.45}
\begin{split}
\left||u_n|^{p_n-2}u_n\psi\right|&\leq
C\left(|u_n|^{\underline{p}-1}|\psi|+|u_n|^{\bar{p}-1}|\psi|\right)\\
&\leq
C\left(|x|^{\frac{2-N}2(\underline{p}-1)}|\psi|+|x|^{\frac{2-N}2(\bar{p}-1)}|\psi|\right)\in
L^{\frac{2N}{N+\alpha}}(\mathbb{R}^N).
\end{split}
\end{equation}
By the Hardy-Littlewood-Sobolev inequality and the Lebesgue
dominated convergence theorem,
\begin{equation*}
\begin{split}
\left|\int_{\mathbb{R}^N}\left(I_{\alpha}\ast|u_n|^{p_n}\right)|u_n|^{p_n-2}u_n\psi
-\int_{\mathbb{R}^N}\left(I_{\alpha}\ast|u_n|^{p_n}\right)|u|^{\bar{p}-2}u\psi\right|=o_n(1).
\end{split}
\end{equation*}
Thus, for any $\psi\in C_c^{\infty}(\mathbb{R}^N)$,
\begin{equation*}
\begin{split}
0&=\langle J_{p_n,q}'(u_n),\psi\rangle \\
&=\int_{\mathbb{R}^N}\nabla
u_n\nabla \psi
+u_n\psi-\mu\int_{\mathbb{R}^N}\left(I_{\alpha}\ast|u_n|^{p_n}\right)|u_n|^{p_n-2}u_n\psi-\lambda\int_{\mathbb{R}^N}|u_n|^{q-2}u_n\psi\\
& \to \int_{\mathbb{R}^N}\nabla u\nabla \psi +
u\psi-\mu\int_{\mathbb{R}^N}\left(I_{\alpha}\ast|u|^{\bar{p}}\right)|u|^{\bar{p}-2}u\psi-\lambda\int_{\mathbb{R}^N}|u|^{q-2}u\psi
\end{split}
\end{equation*}
as $n\to \infty$. That is, $u$ is a solution of (\ref{e1.2222}).

We claim that $u\not\equiv 0$. Suppose by contradiction that
$u\equiv0$. By using $P_{p_n,q}(u_n)=0$, $\int_{\mathbb{R}^N}|u_n|^q=o_n(1)$ and the Young inequality
\begin{equation}\label{e1.14}
|u|^{p_n}\leq
\frac{\bar{p}-p_n}{\bar{p}-\underline{p}}|u|^{\underline{p}}+
\frac{p_n-\underline{p}}{\bar{p}-\underline{p}}|u|^{\bar{p}},
\end{equation}
we get that
\begin{equation}\label{e1.13}
\begin{split}
\int_{\mathbb{R}^N}|\nabla u_n|^2
+\frac{N}{N-2}\int_{\mathbb{R}^N}|u_n|^2&=\frac{\mu(N+\alpha)}{(N-2)p_n}\int_{\mathbb{R}^N}\left(I_{\alpha}\ast
|u_n|^{p_n}\right)|u_n|^{p_n}+o_n(1)\\
&\leq \mu\int_{\mathbb{R}^N}(I_{\alpha}\ast
|u_n|^{\bar{p}})|u_n|^{\bar{p}}+o_n(1)\\
&\leq \mu\left(\frac{\int_{\mathbb{R}^N}|\nabla
u_n|^2}{S_\alpha}\right)^{\bar{p}}+o_n(1),
\end{split}
\end{equation}
which implies that either $\|u_n\|_{H^1(\mathbb{R}^N)}\to 0$ or
$\limsup_{n\to \infty}\|\nabla u_n\|_2^2\geq
\mu^{-\frac{N-2}{2+\alpha}}S_\alpha^{\frac{N+\alpha}{2+\alpha}}$. If
$\|u_n\|_{H^1(\mathbb{R}^N)}\to 0$, then (\ref{e1.36}) and
(\ref{e1.35}) imply that $c_{p_n,q}\to 0$, which contradicts Lemma
\ref{lem xiajixian}. If $\limsup_{n\to \infty}\|\nabla u_n\|_2^2\geq
\mu^{-\frac{N-2}{2+\alpha}}S_\alpha^{\frac{N+\alpha}{2+\alpha}}$, by
using the first equality in (\ref{e1.13}), we obtain that
\begin{equation*}
\begin{split}
c_{\bar{p},q}&\geq \limsup_{n\to \infty}c_{p_n,q}\\
       &=\limsup_{n\to\infty}\left(J_{p_n,q}(u_n)-\frac{1}{N}P_{p_n,q}(u_n)\right)\\
       &=\limsup_{n\to\infty}\left(\frac{1}{N}\int_{\mathbb{R}^N}|\nabla
                     u_n|^2+\frac{\mu\alpha}{2Np_n}\int_{\mathbb{R}^N}(I_{\alpha}\ast|u_n|^{p_n})|u_n|^{p_n}\right)\\
       &\geq \limsup_{n\to\infty}\left(\frac{1}{N}\int_{\mathbb{R}^N}|\nabla
                     u_n|^2+\frac{(N-2)\alpha}{2N(N+\alpha)}\int_{\mathbb{R}^N}|\nabla
                     u_n|^2\right)\\
       &\geq \frac{2+\alpha}{2(N+\alpha)}\mu^{-\frac{N-2}{2+\alpha}}S_\alpha^{\frac{N+\alpha}{2+\alpha}},
\end{split}
\end{equation*}
which contradicts Lemma \ref{lem shangjie}. Thus $u\not\equiv0$. By
Corollary  \ref{cor1.1}, $P_{\bar{p},q}(u)=0$.

By Fatou's lemma, we have
\begin{equation}\label{e1.58}
\begin{split}
c_{\bar{p},q}&\leq J_{\bar{p},q}(u)\\
       &=J_{\bar{p},q}(u)-\frac{1}{N}P_{\bar{p},q}(u)\\
        &=\frac{1}{N}\int_{\mathbb{R}^N}|\nabla
                     u|^2+\frac{\mu\alpha}{2N\bar{p}}\int_{\mathbb{R}^N}(I_{\alpha}\ast|u|^{\bar{p}})|u|^{\bar{p}}\\
       &\leq\liminf_{n\to\infty}\left(\frac{1}{N}\int_{\mathbb{R}^N}|\nabla
                     u_n|^2+\frac{\mu\alpha}{2Np_n}\int_{\mathbb{R}^N}(I_{\alpha}\ast|u_n|^{p_n})|u_n|^{p_n}\right)\\
       &=\liminf_{n\to\infty}\left(J_{p_n,q}(u_n)-\frac{1}{N}P_{p_n,q}(u_n)\right) \\
       &=\liminf_{n\to\infty}c_{p_n,q}\leq \limsup_{n\to\infty}c_{p_n,q}\leq c_{\bar{p},q}.
\end{split}
\end{equation}
Hence $J_{\bar{p},q}(u)=c_{\bar{p},q}$. By the definition of
$c_{\bar{p},q}^g$, we have $c_{\bar{p},q}^g\leq
J_{\bar{p},q}(u)=c_{\bar{p},q}$, which combining with Lemma
\ref{rek1.1} show that
$c_{\bar{p},q}^g=c_{\bar{p},q}=J_{\bar{p},q}(u)$. That is, $u$ is a
nonnegative and radially nonincreasing groundstate solution of
(\ref{e1.2222}). The strongly maximum principle implies that $u$ is
positive. The proof is complete.

\medskip
\textbf{Remark}. Let $N\geq 3$, $\alpha\in (0,N)$,
$p=\frac{N+\alpha}{N-2}$ and $q\in (2,\frac{2N}{N-2})$. Denote the energy functional corresponding to (1.1) with $\lambda>0$ by $J_{p,q,\lambda}(u)$ and define
\begin{equation*}
\begin{split}
\lambda^*:=&\inf\left\{\lambda_1>0 \  |\   \ \mathrm{ (\ref{e1.2222})\ with \
}\lambda>\lambda_1
\   \mathrm{admits \ a \ solution} \ u\in
H^1(\mathbb{R}^N)\setminus \{0\}\right.\\
&\qquad \qquad \qquad\qquad \left.\ \ \mathrm{satisfying\
}J_{p,q,\lambda}(u)<\frac{2+\alpha}{2(N+\alpha)}S_\alpha^{\frac{N+\alpha}{2+\alpha}}
\right\}.
\end{split}
\end{equation*}
By the proof
of Theorem \ref{thm1.3},
$\lambda^*$ is well defined and  $0\leq \lambda^*<+\infty$. Clearly,  $\lambda^*=0$ if
$q\in (2,\frac{2N}{N-2})$ for $N\geq 4$ or $q\in (4,\frac{2N}{N-2})$
for $N=3$. Moreover, we claim that if $\lambda^*>0$,  then (\ref{e1.2222}) with $\lambda\in (0,\lambda^*)$
does not admit a nontrivial solution $u\in H^1(\mathbb{R}^N)$
satisfying
$J_{p,q,\lambda}(u)<\frac{2+\alpha}{2(N+\alpha)}S_\alpha^{\frac{N+\alpha}{2+\alpha}}$.
We assume  by contradiction that (\ref{e1.2222}) with
$\lambda=\lambda_1\in (0,\lambda^*)$ admits a nontrivial solution $u\in
H^1(\mathbb{R}^N)$ satisfying
$J_{p,q,\lambda_1}(u)<\frac{2+\alpha}{2(N+\alpha)}S_\alpha^{\frac{N+\alpha}{2+\alpha}}$.
Then for any $\lambda_2>\lambda_1$, Lemma \ref{lem1.2} implies that
there exists a unique $\tau_0>0$ such that
$P_{p,q,\lambda_2}(u_{\tau_0})=0$ and
$J_{p,q,\lambda_2}(u_{\tau_0})=\max_{\tau\geq
0}J_{p,q,\lambda_2}(u_{\tau})$. Hence,
\begin{equation*}
c_{p,q,\lambda_2}\leq \max_{\tau\geq 0}J_{p,q,\lambda_2}(u_\tau)\leq
\max_{\tau\geq
0}J_{p,q,\lambda_1}(u_\tau)=J_{p,q,\lambda_1}(u)<\frac{2+\alpha}{2(N+\alpha)}S_\alpha^{\frac{N+\alpha}{2+\alpha}},
\end{equation*}
and then by the proof of Theorem \ref{thm1.3}, equation
(\ref{e1.2222}) with $\lambda=\lambda_2$ admits a nontrivial
solution $v\in H^1(\mathbb{R}^N)$ with
$J_{p,q,\lambda_2}(v)<\frac{2+\alpha}{2(N+\alpha)}S_\alpha^{\frac{N+\alpha}{2+\alpha}}$,
which contradicts the definition of $\lambda^*$. Hence, the claim
holds.

\medskip

\textbf{Proof of Theorem \ref{thm1.4}}. Let $\mu=1$, $p_n\to
\underline{p}^{+}$ as $n\to \infty$ and $\{u_n:=u_{p_n,q}\}\subset
H_r^1(\mathbb{R}^N)$ be a positive and radially nonincreasing
sequence which satisfies (\ref{e1.34}). Lemma \ref{lem xiajixian}
shows that $\{u_n\}$ is bounded in $H^1(\mathbb{R}^N)$. Thus, there
exists a nonnegative and radially nonincreasing function $u\in
H_r^1(\mathbb{R}^N)$ such that up to a subsequence,
$u_n\rightharpoonup u$ weakly in $H^1(\mathbb{R}^N)$, $u_n\to u$
strongly in $L^t(\mathbb{R}^N)$ for $t\in (2,\frac{2N}{N-2})$ and
$u_n\to u$ a.e. in $\mathbb{R}^N$. By Lemma \ref{lem1.7}, $p_n\to
\underline{p}^{+}$ and the H\"{o}lder inequality, we have
\begin{equation}\label{e1.633}
\{|u_n|^{p_n}\}\ \textrm{is\ bounded\ in}\
L^{\frac{2N}{N+\alpha}}(\mathbb{R}^N), \ \{|u_n|^{q-2}u_n\}\ \textrm{is\ bounded\ in}\
L^{\frac{q}{q-1}}(\mathbb{R}^N),
\end{equation}
$\{|u_n|^{p_n-2}u_n\}$ is bounded in
$L^{\frac{2N\underline{p}}{(\underline{p}-1)(N+\alpha)}}(\mathbb{R}^N)$,
$\{|u_n|^{p_n-2}u_n\psi\}$ is bounded in
$L^{\frac{2N}{N+\alpha}}(\mathbb{R}^N)$ and
$\{|u|^{\underline{p}-2}u\psi\}\in
L^{\frac{2N}{N+\alpha}}(\mathbb{R}^N)$ for any $\psi\in
C_c^{\infty}(\mathbb{R}^N)$. Similarly to the proof of Theorem
\ref{thm1.3}, $u$ is a solution of (\ref{e1.2222}).

We claim that $u\not\equiv 0$. Suppose by contradiction that
$u\equiv0$. By using $P_{p_n,q}(u_n)=0$, $\int_{\mathbb{R}^N}|u_n|^q=o_n(1)$ and the Young inequality (\ref{e1.14}),
we get that
\begin{equation}\label{e1.1355}
\begin{split}
\int_{\mathbb{R}^N}|\nabla u_n|^2
+\frac{N}{N-2}\int_{\mathbb{R}^N}|u_n|^2&=\frac{\mu(N+\alpha)}{(N-2)p_n}\int_{\mathbb{R}^N}\left(I_{\alpha}\ast
|u_n|^{p_n}\right)|u_n|^{p_n}+o_n(1)\\
&\leq \frac{\mu N}{N-2}\int_{\mathbb{R}^N}(I_{\alpha}\ast
|u_n|^{\underline{p}})|u_n|^{\underline{p}}+o_n(1)\\
&\leq \frac{\mu N}{N-2}\left(\frac{\int_{\mathbb{R}^N}|
u_n|^2}{S_1}\right)^{\underline{p}}+o_n(1),
\end{split}
\end{equation}
which implies that either $\|u_n\|_{H^1(\mathbb{R}^N)}\to 0$ or
$\limsup_{n\to \infty}\|u_n\|_2^2\geq
\mu^{-\frac{N}{\alpha}}S_1^{\frac{N+\alpha}{\alpha}}$. If
$\|u_n\|_{H^1(\mathbb{R}^N)}\to 0$, then (\ref{e1.36}) and
(\ref{e1.35}) imply that $c_{p_n,q}\to 0$, which contradicts Lemma
\ref{lem xiajixian}. If $\limsup_{n\to \infty}\|u_n\|_2^2\geq
\mu^{-\frac{N}{\alpha}}S_1^{\frac{N+\alpha}{\alpha}}$, by using the
first equality in (\ref{e1.1355}), we obtain that
\begin{equation*}
\begin{split}
c_{\underline{p},q}&\geq \limsup_{n\to \infty}c_{p_n,q}\\
       &=\limsup_{n\to\infty}\left(J_{p_n,q}(u_n)-\frac{1}{N}P_{p_n,q}(u_n)\right)\\
       &=\limsup_{n\to\infty}\left(\frac{1}{N}\int_{\mathbb{R}^N}|\nabla
                     u_n|^2+\frac{\mu\alpha}{2Np_n}\int_{\mathbb{R}^N}(I_{\alpha}\ast|u_n|^{p_n})|u_n|^{p_n}\right)\\
       &\geq \limsup_{n\to\infty}\left(\frac{1}{N}\int_{\mathbb{R}^N}|\nabla
                     u_n|^2+\frac{\alpha}{2(N+\alpha)}\int_{\mathbb{R}^N}|
                     u_n|^2\right)\\
       &\geq \frac{\alpha}{2(N+\alpha)}\mu^{-\frac{N}{\alpha}}S_1^{\frac{N+\alpha}{\alpha}},
\end{split}
\end{equation*}
which contradicts Lemma \ref{lem1.8}. Thus $u\not\equiv0$. By
Corollary \ref{cor1.1}, $P_{\underline{p},q}(u)=0$.

Similarly to the proof of Theorem \ref{thm1.3}, $u$ is a positive
and radially nonincreasing groundstate solution of (\ref{e1.2222}).
The proof is complete.

\medskip

\textbf{Proof of Theorem \ref{thm1.5}}. Let $\lambda=1$, $q_n\to
\bar{q}^{-}$ as $n\to \infty$ and $\{u_n:=u_{p,q_n}\}\subset
H_r^1(\mathbb{R}^N)$ be a positive and radially nonincreasing
sequence which satisfies (\ref{e1.34}). Lemma \ref{lem xiajixian}
shows that $\{u_n\}$ is bounded in $H^1(\mathbb{R}^N)$. Thus, there
exists a nonnegative and radially nonincreasing function $u\in
H_r^1(\mathbb{R}^N)$ such that up to a subsequence,
$u_n\rightharpoonup u$ weakly in $H^1(\mathbb{R}^N)$, $u_n\to u$
strongly in $L^t(\mathbb{R}^N)$ for $t\in (2,\frac{2N}{N-2})$ and
$u_n\to u$ a.e. on $\mathbb{R}^N$. Thus, $|u_n|^{p}\to |u|^p$
strongly in $L^{\frac{2N}{N+\alpha}}(\mathbb{R}^N)$,
$|u_n|^{p-2}u_n\to |u|^{p-2}u$ strongly in
$L^{\frac{2Np}{(p-1)(N+\alpha)}}(\mathbb{R}^N)$ and
$|u_n|^{p-2}u_n\psi\to |u|^{p-2}u\psi$ strongly in
$L^{\frac{2N}{N+\alpha}}(\mathbb{R}^N)$ for any $\psi\in
C_c^{\infty}(\mathbb{R}^N)$. By Remark \ref{rek1.3}, $I_\alpha\ast
|u_n|^{p}\to I_\alpha\ast |u|^{p}$ strongly in
$L^{\frac{2N}{N-\alpha}}(\mathbb{R}^N)$.

It follows from $N\geq 3$ that $\frac{N}{\frac{N-2}2(2-1)}$ and
$\frac{N}{\frac{N-2}2(\bar{q}-1)} \in (1,\infty)$. Since  $q_n\to
\bar{q}^{-}$ and $\psi \in L^{r}(\mathbb{R}^N)$ for $r\in
(1,\infty)$, by the Young inequality, the H\"{o}lder inequality and
Lemma \ref{lem jx} with $t=2N/(N-2)$, there exists a constant $C>0$
independent of $n$ such that
\begin{equation}\label{e1.52}
\begin{split}
\left||u_n|^{q_n-2}u_n\psi\right|&\leq
C\left(|u_n|^{2-1}|\psi|+|u_n|^{\bar{q}-1}|\psi|\right)\\
&\leq
C\left(|x|^{\frac{2-N}2(2-1)}|\psi|+|x|^{\frac{2-N}2(\bar{q}-1)}|\psi|\right)\in
L^{1}(\mathbb{R}^N).
\end{split}
\end{equation}
By the Hardy-Littlewood-Sobolev inequality and the Lebesgue
dominated convergence theorem, we have
\begin{equation*}
\begin{split}
0&=\langle J_{p,q_n}'(u_n),\psi\rangle \\
&=\int_{\mathbb{R}^N}\nabla u_n\nabla \psi
+u_n\psi-\mu\int_{\mathbb{R}^N}\left(I_{\alpha}\ast|u_n|^{p}\right)|u_n|^{p-2}u_n\psi-\lambda\int_{\mathbb{R}^N}|u_n|^{q_n-2}u_n\psi\\
& \to \int_{\mathbb{R}^N}\nabla u\nabla \psi +
u\psi-\mu\int_{\mathbb{R}^N}\left(I_{\alpha}\ast|u|^{p}\right)|u|^{p-2}u\psi-\lambda\int_{\mathbb{R}^N}|u|^{\bar{q}-2}u\psi
\end{split}
\end{equation*}
as $n\to \infty$. That is, $u$ is a solution of (\ref{e1.222}).

We claim that $u\not\equiv 0$. Suppose by contradiction that
$u\equiv0$. By using $P_{p,q_n}(u_n)=0$, $\int_{\mathbb{R}^N}\left(I_{\alpha}\ast|u_n|^{p}\right)|u_n|^{p}=o_n(1)$ and the Young inequality
\begin{equation}\label{e1.53}
|u|^{q_n}\leq \frac{\bar{q}-q_n}{\bar{q}-2}|u|^{2}+
\frac{q_n-2}{\bar{q}-2}|u|^{\bar{q}},
\end{equation}
we get that
\begin{equation}\label{e1.54}
\begin{split}
\frac{N-2}{2}\int_{\mathbb{R}^N}|\nabla u_n|^2
+\frac{N}{2}\int_{\mathbb{R}^N}|u_n|^2&=\frac{\lambda N}{q_n}\int_{\mathbb{R}^N}|u_n|^{q_n}+o_n(1)\\
&\leq \frac{\lambda
N}{\bar{q}}\int_{\mathbb{R}^N}|u_n|^{\bar{q}}+o_n(1)\\
&\leq \frac{\lambda
N}{\bar{q}}\left(\frac{\int_{\mathbb{R}^N}|\nabla
u_n|^2}{S}\right)^{\frac{N}{N-2}}+o_n(1),
\end{split}
\end{equation}
which implies that either $\|u_n\|_{H^1(\mathbb{R}^N)}\to 0$ or
$\limsup_{n\to \infty}\|\nabla u_n\|_2^2\geq
\lambda^{-\frac{N-2}{2}}S^{\frac{N}{2}}$. If
$\|u_n\|_{H^1(\mathbb{R}^N)}\to 0$, then (\ref{e1.36}) and
(\ref{e1.35}) imply that $c_{p,q_n}\to 0$, which contradicts Lemma
\ref{lem xiajixian}. If $\limsup_{n\to \infty}\|\nabla u_n\|_2^2\geq
\lambda^{-\frac{N-2}{2}}S^{\frac{N}{2}}$, then
\begin{equation}\label{e1.62}
\begin{split}
c_{p,\bar{q}}&\geq \limsup_{n\to \infty}c_{p,q_n}\\
       &=\limsup_{n\to\infty}\left(J_{p,q_n}(u_n)-\frac{1}{N}P_{p,q_n}(u_n)\right)\\
       &=\limsup_{n\to\infty}\left(\frac{1}{N}\int_{\mathbb{R}^N}|\nabla
                     u_n|^2+\frac{\mu\alpha}{2Np}\int_{\mathbb{R}^N}(I_{\alpha}\ast|u_n|^{p})|u_n|^{p}\right)\\
       &\geq \limsup_{n\to\infty}\frac{1}{N}\int_{\mathbb{R}^N}|\nabla
                     u_n|^2\\
       &\geq \frac{1}{N}\lambda^{-\frac{N-2}{2}}S^{\frac{N}{2}},
\end{split}
\end{equation}
which contradicts Lemma \ref{lem1.9}. Thus $u\not\equiv0$. By
Corollary  \ref{cor1.1}, $P_{p,\bar{q}}(u)=0$.

Similarly to the proof of Theorem \ref{thm1.3}, $u$ is a positive
and radially nonincreasing groundstate solution of (\ref{e1.222}).
The proof is complete.

\bigskip

\textbf{Proof of Theorem \ref{thm1.6}}. Let $a_n\to 0^{+}$ as $n\to
\infty$ and $\{u_n:=u_{\underline{p}+a_n,\bar{q}-a_n}\}\subset
H_r^1(\mathbb{R}^N)$ be a positive and radially nonincreasing sequence
which satisfies (\ref{e1.34}). Lemma \ref{lem xiajixian} shows that $\{u_n\}$ is bounded
in $H^1(\mathbb{R}^N)$. Thus, there exists a nonnegative and radially
nonincreasing function $u\in H_r^1(\mathbb{R}^N)$ such that up to a
subsequence, $u_n\rightharpoonup u$ weakly in $H^1(\mathbb{R}^N)$,
$u_n\to u$ strongly in $L^t(\mathbb{R}^N)$ for $t\in
(2,\frac{2N}{N-2})$ and $u_n\to u$ a.e. in $\mathbb{R}^N$. Similarly to the proof of Theorem \ref{thm1.4} and \ref{thm1.5},
 $u$ is a solution of (\ref{e1.22222}).

We claim that $u\not\equiv 0$. Suppose by contradiction that
$u\equiv0$. By using\\ $P_{\underline{p}+a_n,\bar{q}-a_n}(u_n)=0$
and the Young inequality (\ref{e1.14}) and (\ref{e1.53}), we get
that
\begin{equation}\label{e1.135}
\begin{split}
\frac{N-2}{2}\int_{\mathbb{R}^N}&|\nabla u_n|^2
+\frac{N}{2}\int_{\mathbb{R}^N}|u_n|^2\\
&=\frac{\mu(N+\alpha)}{2(\underline{p}+a_n)}\int_{\mathbb{R}^N}\left(I_{\alpha}\ast
|u_n|^{\underline{p}+a_n}\right)|u_n|^{\underline{p}+a_n}+\frac{\lambda N}{\bar{q}-a_n}\int_{\mathbb{R}^N}|u_n|^{\bar{q}-a_n}\\
&\leq \frac{\mu(N+\alpha)}{2\underline{p}}\int_{\mathbb{R}^N}(I_{\alpha}\ast
|u_n|^{\underline{p}})|u_n|^{\underline{p}}+\frac{\lambda N}{\bar{q}}\int_{\mathbb{R}^N}|u_n|^{\bar{q}}+o_n(1)\\
&\leq \frac{\mu(N+\alpha)}{2\underline{p}}\left(\frac{\int_{\mathbb{R}^N}|
u_n|^2}{S_1}\right)^{\underline{p}}+\frac{\lambda N}{\bar{q}}\left(\frac{\int_{\mathbb{R}^N}|\nabla
u_n|^2}{S}\right)^{\frac{N}{N-2}}+o_n(1),
\end{split}
\end{equation}
which implies that either $\|u_n\|_{H^1(\mathbb{R}^N)}\to 0$  or
$\limsup_{n\to \infty}\|\nabla u_n\|_2^2\geq
\lambda^{-\frac{N-2}{2}}S^{\frac{N}{2}}$ or
$\limsup_{n\to \infty}\|u_n\|_2^2\geq
\mu^{-\frac{N}{\alpha}}S_1^{\frac{N+\alpha}{\alpha}}$. If
$\|u_n\|_{H^1(\mathbb{R}^N)}\to 0$, then (\ref{e1.36}) and
(\ref{e1.35}) imply that $c_{\underline{p}+a_n,\bar{q}-a_n}\to 0$, which contradicts Lemma
\ref{lem xiajixian}. If $\limsup_{n\to \infty}\|\nabla u_n\|_2^2\geq
\lambda^{-\frac{N-2}{2}}S^{\frac{N}{2}}$, then similarly to (\ref{e1.62}),
\begin{equation*}
\begin{split}
c_{\underline{p},\bar{q}}&\geq \limsup_{n\to \infty}c_{\underline{p}+a_n,\bar{q}-a_n}\\
       &=\limsup_{n\to\infty}\left(J_{\underline{p}+a_n,\bar{q}-a_n}(u_n)-\frac{1}{N}P_{\underline{p}+a_n,\bar{q}-a_n}(u_n)\right)\\
       &=\limsup_{n\to\infty}\left(\frac{1}{N}\int_{\mathbb{R}^N}|\nabla
                     u_n|^2+\frac{\mu\alpha}{2N(\underline{p}+a_n)}\int_{\mathbb{R}^N}(I_{\alpha}\ast|u_n|^{\underline{p}+a_n})|u_n|^{\underline{p}+a_n}\right)\\
       &\geq \frac{1}{N}\lambda^{-\frac{N-2}{2}}S^{\frac{N}{2}},
\end{split}
\end{equation*}
which contradicts Lemma \ref{lem1.10}. If $\limsup_{n\to
\infty}\|u_n\|_2^2\geq
\mu^{-\frac{N}{\alpha}}S_1^{\frac{N+\alpha}{\alpha}}$, by using\\
$P_{\underline{p}+a_n,\bar{q}-a_n}(u_n)=0$ and $\langle
J'_{\underline{p}+a_n,\bar{q}-a_n}(u_n),u_n\rangle=0$, we obtain
that
\begin{equation*}
\begin{split}
&\frac{\mu\alpha}{2N(\underline{p}+a_n)}\int_{\mathbb{R}^N}(I_{\alpha}\ast|u_n|^{\underline{p}+a_n})|u_n|^{\underline{p}+a_n}\\
&\quad \quad \quad=\frac{\frac{\mu\alpha}{2N(\underline{p}+a_n)}\left(\frac{N}{2}-\frac{N}{\bar{q}-a_n}\right)}{\frac{\mu(N+\alpha)}{2(\underline{p}+a_n)}-\frac{\mu N}{\bar{q}-a_n}}
\int_{\mathbb{R}^N}|u_n|^2+\frac{\frac{\mu\alpha}{2N(\underline{p}+a_n)}\left(\frac{N-2}{2}-\frac{N}{\bar{q}-a_n}\right)}{\frac{\mu(N+\alpha)}{2(\underline{p}+a_n)}-\frac{\mu N}{\bar{q}-a_n}}
\int_{\mathbb{R}^N}|\nabla u_n|^2
\end{split}
\end{equation*}
and then
\begin{equation*}
\begin{split}
c_{\underline{p},\bar{q}}&\geq \limsup_{n\to\infty}\left(\frac{1}{N}\int_{\mathbb{R}^N}|\nabla
                     u_n|^2+\frac{\mu\alpha}{2N(\underline{p}+a_n)}\int_{\mathbb{R}^N}(I_{\alpha}\ast|u_n|^{\underline{p}+a_n})|u_n|^{\underline{p}+a_n}\right)\\
       &\geq \limsup_{n\to\infty}\left(\frac{\frac{\mu\alpha}{2N(\underline{p}+a_n)}\left(\frac{N}{2}-\frac{N}{\bar{q}-a_n}\right)}{\frac{\mu(N+\alpha)}{2(\underline{p}+a_n)}-\frac{\mu N}{\bar{q}-a_n}}
\int_{\mathbb{R}^N}|u_n|^2\right)\\
&\geq \frac{\alpha}{2(N+\alpha)}\mu^{-\frac{N}{\alpha}}S_1^{\frac{N+\alpha}{\alpha}},
\end{split}
\end{equation*}
which contradicts Lemma \ref{lem1.10}. Thus $u\not\equiv0$. By
Corollary \ref{cor1.1}, $P_{\underline{p},\bar{q}}(u)=0$.

Similarly to the proof of Theorem \ref{thm1.3}, $u$ is a positive
and radially nonincreasing groundstate solution of (\ref{e1.22222}).
The proof is complete.

\medskip

\textbf{Acknowledgements.} This work was supported by Tianjin
Municipal Education Commission with the Grant no. 2017KJ173
``Qualitative studies of solutions for two kinds of nonlocal
elliptic equations" and the National Natural Science Foundation of
China (No. 11571187).


\begin{thebibliography}{00}
%
%

\bibitem{Alves 1996} C.O. Alves, D.C. de Morais Filho, M.A.S. Souto, Radially symmetric solutions for a class of critical exponent elliptic
problems in $\mathbb{R}^N$, Electron. J. Differential Equations, 1996(07) (1996), 1-12.

\bibitem{Alves-Souto-Montenegro 2012} C.O. Alves, M.A.S. Souto, M. Montenegro,  Existence of a ground
state solution for a nonlinear scalar field equation with critical
growth, Calc. Var. PDE,  43 (2012), 537-554.


\bibitem{Ao 2016} Yong Ao, Existence of solutions for a class of nonlinear Choquard equations with critical growth, arXiv preprint arXiv:1608.07064, (2016).




\bibitem{Berestycki-Lions 1983}  H. Berestycki, P.L. Lions, Nonlinear scalar field
equations, I Existence of a ground state, Arch. Ration. Mech. Anal.,
82 (1983), 313-345.

\bibitem{Bogachev 2007} V.I. Bogachev, Measure theory, Springer, Berlin, 2007.

\bibitem{Brezis 2012} H. Brezis, Functional analysis, Sobolev spaces and partial
differential equations, Universitext, Springer, New York, 2011.




\bibitem{Brezis and Kato 1979}  H. Brezis, T. Kato, Remarks on the Schr\"{o}dinger operator
with singular complex potentials, J. Math. Pures Appl., 58
(1979), 137-151.



\bibitem{Brezis-Nirenberg 1983} H. Brezis, L. Nirenberg,  Positive solutions of nonlinear
elliptic equations involving critical Sobolev exponents, Commun.
Pure Appl. Math., 36 (1983), 437-477.



\bibitem{Gao-Yang-1} F. Gao, M. Yang, On the Brezis-Nirenberg type critical problem
for nonlinear Choquard equation, SCIENCE CHINA Mathematics, doi:
10.1007/s11425-000-0000-0.

\bibitem{Gao-Yang-2} F. Gao, M. Yang, On nonlocal Choquard equations with
Hardy-Littlewood-Sobolev critical exponents, Journal of Mathematical
Analysis and Applications, 448(2) (2017), 1006-1041.

\bibitem{Gross 1996} E.P. Gross, Physics of many-Particle systems, Gordon Breach,
New York, Vol.1, 1996.



\bibitem{Li-Ma-Zhang} Xinfu Li, Shiwang Ma, Guang Zhang, Existence and  qualitative
properties of solutions for  Choquard equations with a local term,
Nonlinear Analysis: Real World Applications, 45 (2019), 1-25.

\bibitem{Lieb 1977}  E.H. Lieb, Existence and uniqueness of the minimizing solution of
Choquard's nonlinear equation, Stud. Appl. Math., 57(2) (1977),
93-105.

\bibitem{Lieb-Loss 2001} E.H. Lieb, M. Loss, Analysis, volume 14 of graduate studies in
mathematics, American Mathematical Society, Providence, RI, (4)
2001.

\bibitem{Liu-Liao-Tang 2017} Jiu Liu, JiaFeng Liao, ChunLei Tang, Ground state solution for a
class of Schr\"{o}dinger equations involving general critical growth
term, Nonlinearity, 30 (2017), 899-911.



\bibitem{Moroz-Schaftingen JFA 2013} V. Moroz, J. Van Schaftingen, Groundstates of nonlinear Choquard equations: existence, qualitative properties and decay asymptotics, Journal of Functional Analysis, 265(2) (2013), 153-184.


\bibitem{Moroz-Schaftingen 2015} V. Moroz, J. Van Schaftingen, Existence of groundstates for a class of nonlinear Choquard
equations, Trans. Amer. Math. Soc., 367 (2015), 6557-6579.

\bibitem{Moroz-Schaftingen 2017} V. Moroz, J. Van Schaftingen, A guide to the Choquard equation, Journal of Fixed Point Theory and Applications, 19(1) (2017), 773-813.


\bibitem{Pekar 1954} S. Pekar, Untersuchung ¨¹ber die Elektronentheorie der
Kristalle, Akademie Verlag, Berlin, 1954.

\bibitem{Penrose 1996} R. Penrose, On gravity's role in quantum state reduction, Gen.
Rel. Grav., 28 (1996), 581-600.


\bibitem{Riesz1949AM} M. Riesz, L'int\'{e}grale de Riemann-Liouville et le probl\`{e}me de Cauchy, Acta Math., 81 (1949), 1-223.

\bibitem{Schaftingen-Xia 2017}  J. Van Schaftingen, J. Xia, Groundstates for a local
nonlinear perturbation of the Choquard equations with lower critical
exponent, arXiv preprint arXiv:1710.03973, (2017).

\bibitem{Seok AML 2017} J. Seok, Nonlinear Choquard equations involving a critical local term, Applied Mathematics Letters, 63 (2017), 77-87.


\bibitem{Seok 2018} J. Seok, Nonlinear Choquard equations: Doubly critical case,
Applied Mathematics Letters, 76 (2018), 148-156.

\bibitem{Struwe 2008} M. Struwe, Variational methods, Berlin etc., Springer, Fourth
edition, 2008.



\bibitem{Willem 1996} M. Willem, Minimax Theorems, Birkh\"{a}user, Boston, 1996.


\bibitem{Willem 2013} M. Willem,  Functional analysis: Fundamentals and
applications, Cornerstones, Vol.XIV, Birkh\"{a}user, Basel, 2013.


\bibitem{Zhang-Zou 2012} J.J. Zhang, W.M. Zou,  A Berestycki-Lions theorem
revisited, Commun. Contemp. Math., 14 (2012), 1250033.

\bibitem{Zhang-Zou 2014}  J. Zhang,  W.M. Zou,  The critical case for a Berestycki-Lions
theorem, Sci. China Math., 57 (2014), 541-554.
\end{thebibliography}


\end{document}